\documentclass[10pt,twoside]{article}
\usepackage[hmarginratio=1:1,a4paper,text={16cm,23cm}]{geometry}
\usepackage[english]{babel}
\usepackage[utf8]{inputenc}
\usepackage[T1]{fontenc}
\usepackage{amsmath,amsfonts,amssymb,amsthm}
\usepackage{tikz}
	\usetikzlibrary{arrows,calc,decorations.text,decorations.pathreplacing}
\usepackage{dsfont,color,graphicx,mathrsfs,enumerate,pdfpages,setspace}
\usepackage{caption}
\usepackage{url}
\usepackage[backref=page]{hyperref}
\hypersetup{allcolors=blue,colorlinks=true,breaklinks=true,bookmarksopen=true}

\usepackage{amsthm}
\newcounter{count}[section]\numberwithin{count}{section}
\newtheorem{theorem}[count]{Theorem}

\newtheorem{proposition}[count]{Proposition}

\newtheorem{lemma}[count]{Lemma}
\newtheorem{assumption}[count]{Assumption}
\numberwithin{equation}{section}

\makeatletter
\renewcommand{\thefigure}{\ifnum \c@section>\z@ \thesection.\fi
 \@arabic\c@figure}
\@addtoreset{figure}{section}
\makeatother

\makeatletter
\renewenvironment{proof}[1][\proofname]{\par
  \pushQED{\qed}%
  \normalfont \topsep6\p@\@plus6\p@\relax
  \trivlist
  \item[\hskip\labelsep
        \bfseries
    #1\@addpunct{\scantokens{:}}]\ignorespaces
}{%
  \popQED\endtrivlist\@endpefalse
}
\makeatother

\title{Ergodicity of inhomogeneous Markov chains through asymptotic pseudotrajectories}

\author{Michel \textsc{Benaïm}, Florian \textsc{Bouguet}, Bertrand \textsc{Cloez}}

\date{
    \emph{Université de Neuchâtel}\\
    \emph{Inria team BIGS, IECL}\\
    \emph{INRA-SupAgro MISTEA}\\[2ex]
    September 15, 2016
}

\usepackage{fancyhdr}
		\makeatletter
			\let\runauthor\@author
			\let\runtitle\@title
		\makeatother
        \fancypagestyle{main}{
                \fancyhf{}
                \fancyhead[CE]{\runtitle}
                \fancyhead[CO]{\runauthor}
                \fancyfoot[c]{\thepage}
                
                }
        \fancypagestyle{plain}{
                \fancyhf{}
                \fancyfoot[c]{\thepage}
                
                }

\newenvironment{remark}[1][]{\par\noindent\refstepcounter{count}\textbf{Remark}
\ifx\newenvironment#1\newenvironment
	\textbf{\arabic{section}.\arabic{count}.}
\else
	\textbf{\arabic{section}.\arabic{count}} (#1)\textbf{.}
\fi}{\leavevmode\unskip\penalty9999 \hbox{}\nobreak\hfill\quad\hbox{{$\diamondsuit$}}}
\newenvironment{acknowledgements}{\noindent\textbf{Acknowledgements:}}{}

\definecolor{darkred}{rgb}{0.9,0.1,0.1}

\newcommand*{\e}{\text{e}}
\newcommand*{\E}{\mathbb{E}}

\newcommand*{\N}{\mathbb{N}}
\newcommand*{\prob}{\mathbb{P}}

\newcommand*{\R}{\mathbb{R}}

\newcommand*{\indic}{\mathds{1}}
\newcommand*{\Wass}{W}

\newcommand*{\sgn}{\text{sgn}}

\begin{document}

\pagestyle{main}
\maketitle
\begin{center}
\begin{minipage}[c]{.8\textwidth}
\textbf{Abstract:} In this work, we consider an inhomogeneous (discrete time) Markov chain and are interested in its long time behavior. We provide sufficient conditions to ensure that some of its asymptotic properties can be related to the ones of a homogeneous (continuous time) Markov process. Renowned examples such as a bandit algorithms, weighted random walks or decreasing step Euler schemes are included in our framework. Our results are related to functional limit theorems, but the approach differs from the standard "Tightness/Identification" argument; our method is unified and based on the notion of pseudotrajectories on the space of probability measures.

\tableofcontents

\vspace{1cm}


\end{minipage}
\end{center}

\section{Introduction}\label{section:intro}
In this paper, we consider an inhomogeneous Markov chain $(y_n)_{n\geq 0}$ on $\R^D$, and a non-increasing sequence $(\gamma_n)_{n\geq1}$ converging to 0, such that $\sum_{n=1}^\infty\gamma_n=+\infty$. For any smooth function $f$, we set
\begin{equation}
\mathcal L_nf(y):=\frac{\E\left[f(y_{n+1})-f(y_n)|y_n=y\right]}{\gamma_{n+1}}.
\label{eq:DefLn}
\end{equation}
We shall establish general asymptotic results when $\mathcal{L}_n$ converges, in some sense explained below, toward some infinitesimal generator $\mathcal{L}$. We prove that, under reasonable hypotheses, one can deduce properties (trajectories, ergodicity, etc) of $(y_n)_{n \geq1}$ from the ones of a process generated by~$\mathcal{L}$.

This work is mainly motivated by the study of the rescaling of stochastic approximation algorithms (see e.g. \cite{Ben99,LP13}). Classically, such rescaled algorithms converge to Normal distributions (or linear diffusion processes); see e.g. \cite{Duf96,KY03,For15}. This central limit theorem is usually proved with the help of "Tightness/Identification" methods. With the same structure of proof, Lamberton and Pagès get a different limit in \cite{LP08}; namely, they provide a convergence to the stationary measure of a non-diffusive Markov process. Closely related, the decreasing step Euler scheme (as developed in \cite{LP02,Lem05}) behaves in the same way.

In contrast to this classical approach, we rely on the notion of asymptotic pseudotrajectories introduced in \cite{BH96}. Therefore, we focus on the asymptotic behavior of $\mathcal L_n$ using Taylor expansions to deduce immediately the form of a limit generator $\mathcal L$. A natural way to understand the asymptotic behavior of $(y_n)_{n\geq0}$ is to consider it as an approximation of a Markov process generated by $\mathcal L$. Then, provided that the limit Markov process is ergodic and that we can estimate its speed of convergence toward the stationary measure, it is natural to deduce convergence and explicit speeds of convergence of $(y_n)_{n\geq0}$ toward equilibrium. Our point of view can be related to the Trotter-Kato theorem (see e.g. \cite{Kal02}). The proof of our main theorem, Theorem~\ref{theorem:APTrPT} below, is related to Lindeberg's proof of the central limit theorem; namely it is based on a telescopic sum and a Taylor expansion.

With the help of Theorem~\ref{theorem:APTrPT}, the study of the long time behavior of $(y_n)_{n\geq0}$ reduces to the one of a homogeneous-time Markov process. Their convergence has been widely studied in the litterature, and we can differentiate several approaches. For instance, there are so-called "Meyn-and-Tweedie" methods (or Foster-Lyapunov criteria, see \cite{MT93b,HM11,HMS11,CH15}) which provide qualitative convergence under mild conditions; we can follow this approach to provide qualitative properties for our inhomogeneous Markov chain. However, the speed is usually not explicit or very poor. Another approach consists in the use of \textit{ad hoc} coupling methods (see e.g. \cite{Lin92,Ebe11,Bou15}) either for a diffusion or a piecewise deterministic Markov process (PDMP). Those methods usually prove themselves to be efficient for providing explicit speeds of convergence, but rely on extremely particular strategies. Among other approaches, let us also mention functional inequalities or spectral gap methods (see e.g. \cite{Bak94,ABC+00,Clo12,Mon14}).

In this article, we develop a unified approach to study the long time behavior of inhomogeneous Markov chains, which may also provide speeds of convergence or functional convergence. To our knowledge, this method is original, and Theorems~\ref{theorem:APTrPT} and \ref{theorem:SpeedConv} have the advantage of being self-contained. The main goal of our illustrations, in Section~\ref{section:illustrations}, is to provide a simple framework to understand our approach. For these examples, proofs seem more simple and intuitive, and we are able to recover classical results as well as slight improvements.

This paper is organized as follows. In Section~\ref{section:results}, we state the framework and the main assumptions that will be used throughout the paper. We recall the notion of asymptotic pseudotrajectory, and present our main result, Theorem~\ref{theorem:APTrPT}, which describes the asymptotic behavior of a Markov chain. We also provide two consequences, Theorems~\ref{theorem:SpeedConv} and \ref{theorem:functionalCV}, precising the geometric ergodicity of the chain or its functional convergence. In Section~\ref{section:illustrations}, we illustrate our results by showing how some renowned examples, including weighted random walks, bandit algorithms or decreasing step Euler schemes, can be easily studied with this unified approach. In Section~\ref{section:proofs} and \ref{section:appendix}, we provide the proofs of our main theorems and of the technical parts left aside while dealing with the illustrations.

\section{Main results}\label{section:results}
\subsection{Framework}\label{subsection:framework}
We shall use the following notation in the sequel:
\begin{itemize}
	\item A multi-index is a $D$-tuple $N=(N_1,\dots,N_D)\in \N^D$; we define the order $N\leq\widetilde N$ if, for all $1\leq i\leq D,N_i\leq\widetilde N_i$. We define $|N|=\sum_{i=1}^DN_i$ and and we identify an integer $N$ with the multi-index $(N,\dots,N)$.
	\item For some multi-index $N$, $\mathscr C^N$ is the set of functions $f:\R^D\to\R$ which are $N_i$ times continuously differentiable in the direction $i$. For any $f\in\mathscr C^N(\R^D),$ we define
	\[f^{(N)}=\frac{\partial^{|N|}}{\partial_{x_1}^{N_1}\dots\partial_{x_D}^{N_D}}f,\quad\|f^{(N)}\|_\infty=\sup_{x\in\R^D}|f^{(N)}(x)|.\]
	\item $\mathscr C^N_b$ is the set of $\mathscr C^N$ functions such that $\sum_{j\leq N}\|f^{(j)}\|_{\infty}<+\infty$. Also, $\mathscr C^N_c$ is the set of $\mathscr C^N$ functions with compact support, and $\mathscr C_0^N$ is the set of $\mathscr C^N$ functions such that $\lim_{\|x\|\to\infty}f(x)=0.$
	\item $\mathscr L(X)$ is the law of a random variable $X$ and $\text{Supp}(\mathscr L(X))$ its support.
	\item $x\wedge y:=\min(x,y)$ and $x\vee y:=\max(x,y)$ for any $x,y\in\R$.
	\item For some multi index $N,\chi_N(x):=\sum_{i=1}^D\sum_{k=0}^{N_i}|x_i|^k$ for $x\in\R^D$.
\end{itemize}

Let us recall some basics about Markov processes. Given a homogeneous Markov process $(X_t)_{t\geq0}$ with càdlàg trajectories a.s., we define its Markov semigroup $(P_t)_{t\geq0}$ by
\[P_tf(x)=\E[f(X_t)\ |\ X_0=x].\]
It is said to be Feller if, for all $f\in\mathscr C_0^0$, $P_tf\in\mathscr C_0^0$ and $\lim_{t\to0}\|P_tf-f\|_\infty=0$. We can define its generator $\mathcal L$ acting on functions $f$ satisfying $\lim_{t\to0}\|t^{-1}(P_tf-f)-\mathcal Lf\|_\infty=0$. The set of such functions is denoted by $\mathcal D(\mathcal L)$, and is dense in $\mathscr C_0^0$; see for instance \cite{EK86}. The semigroup property of $(P_t)$ ensures the existence of a semiflow
\begin{equation}
\Phi(\nu, t):=\nu P_t,
\label{eq:DefPhit}\end{equation}
defined for any probability measure $\nu$ and $t\geq0$; namely, for all $s,t>0$, $\Phi(\nu,t+s)=\Phi(\Phi(\nu,t),s)$.

Let $(y_n)_{n\geq 0}$ be a (inhomogeneous) Markov chain and let $(\mathcal L_n)_{n\geq 0}$ be a sequence of operators satisfying, for $f\in\mathscr C^0_b$,
\[\mathcal L_nf(y_n):=\frac{\E\left[f(y_{n+1})-f(y_n)|y_n\right]}{\gamma_{n+1}},\]
where $(\gamma_n)_{n\geq1}$ is a decreasing sequence converging to 0, such that $\sum_{n=1}^\infty\gamma_n=+\infty$. Note that the sequence $(\mathcal L_n)$ exists thanks to Doob's lemma. Let $(\tau_n)$ be the sequence defined by $\tau_0:=0$ and $\tau_n:=\sum_{k=1}^n \gamma_k$, and let $m(t):=\sup\{n \geq 0 \ : \ t \geq \tau_n\}$ be the unique integer such that $\tau_{m(t)}\leq t<\tau_{m(t)+1}$. We denote by $(Y_t)$ the process defined by $Y_t:= y_n$ when $t\in [\tau_n, \tau_{n+1})$ and we set
\begin{equation}
\mu_t:=\mathscr{L}(Y_t).
\label{eq:DefMut}\end{equation}

Following \cite{BH96,Ben99}, we say that $(\mu_t)_{t\geq0}$ is an asymptotic pseudotrajectory of $\Phi$ (with respect to a distance $d$ over probability distributions) if, for any $T>0$,
\begin{equation}
\lim_{t \to \infty} \sup_{0 \leq s \leq T} d(\mu_{t+s},\Phi(\mu_t, s))=0.
\label{eq:DefAPT}
\end{equation}
Likewise, we say that $(\mu_t)_{t\geq0}$ is a $\lambda$-pseudotrajectory of $\Phi$ (with respect to $d$) if there exists $\lambda>0$ such that, for all $T>0$,
\begin{equation}
\limsup_{t\to+\infty}\frac1t\log\left(\sup_{0\leq s\leq T}d(\mu_{t+s},\Phi(\mu_t, s))\right)\leq-\lambda.
\label{eq:DefrPT}\end{equation}
This definition of $\lambda$-pseudotrajectories is the same as in \cite{Ben99}, up to the sign of $\lambda$.

In the sequel, we discuss asymptotic pseudotrajectories with distances of the form
\[d_{\mathscr F}(\mu,\nu):=\sup_{f \in \mathscr F}\left|\mu(f)-\nu(f)\right|=\sup_{f \in \mathscr F}\left|\int fd\mu-\int fd\nu\right|,\]
for a certain class of functions $\mathscr F$. In particular, this includes total variation, Fortet-Mourier and Wasserstein distances. In general, $d_{\mathscr F}$ is a pseudodistance. Nevertheless, it is a distance whenever $\mathscr F$ contains an algebra of bounded continuous functions that separates points (see \cite[Theorem~4.5.(a), Chapter~3]{EK86}). In all the cases considered here, $\mathscr F$ contains the algebra $\mathscr C^\infty_c$ and then convergence in $d_{\mathscr F}$ entails convergence in distribution (see Lemma~\ref{lemma:WeakConv}, whose proof is classical and is given in the appendix for the sake of completeness).

\subsection{Assumptions and main theorem}\label{subsection:assumptions}

In the sequel, let $d_1,N_1,N_2$ be multi-indices, parameters of the model. We will assume, without loss of generality, that $N_1\leq N_2$. Some key methods of how to check every assumption are provided in Section~\ref{section:illustrations}.

The first assumption we need is crucial. It defines the asymptotic homogeneous Markov process ruling the asymptotic behavior of $(y_n)$.

\begin{assumption}[Convergence of generators]
There exists a non-increasing sequence $(\epsilon_n)_{n \geq1}$ converging to 0 and a constant $M_1$ (depending on $\mathscr L(y_0)$) such that, for all $f \in \mathcal D(\mathcal L)\cap\mathscr C^{N_1}_b$ and $n\in\N^\star$, and for any $y\in\text{Supp}(\mathscr L(y_n))$
\[\left| \mathcal{L}f(y) - \mathcal{L}_n f(y) \right| \leq M_1\chi_{d_1}(y) \sum_{j=0}^{N_1} \| f^{(j)} \|_\infty\mathcal \epsilon_n.\]
\label{assumption:convergence}\end{assumption}

The following assumption is quite technical, but turns out to be true for most of the limit semigroups we deal with. Indeed, this is shown for large classes of PDMPs in Proposition~\ref{proposition:regPDMP} and for some diffusion processes in Lemma~\ref{lemma:DerivativesDiffusion}.

\begin{assumption}[Regularity of the limit semigroup]
For all $T>0$, there exists a constant $C_T$ such that, for every $t\leq T,|j|\leq N_1$ and $f \in \mathscr C^{N_2}_b$,
\[P_tf\in\mathscr C_b^{N_1},\quad |(P_tf)^{(j)}(y)| \leq C_T \sum_{i=0}^{N_2}  \| f^{(i)} \|_\infty.\]
\label{assumption:regularity}\end{assumption}

The next assumption is a standard condition of uniform boundedness of the moments of the Markov chain. We also provide a very similar Lyapunov criterion to check this condition.

\begin{assumption}[Uniform boundedness of moments]
\label{assumption:moments}
Assume that there exists a multi-index $d\geq d_1$ such that one of the following statements holds:
\begin{enumerate}[i)]
	\item There exists a constant $M_2$ (depending on $\mathscr L(y_0)$) such that
\[\sup_{n\geq 0} \E[\chi_d(y_n)] \leq M_2.\]
	\item There exists $V:\R^D\to\R_+$ such that, for all $n\geq0$, $\E[V(y_n)]<+\infty$. Moreover, there exist $n_0\in\N^\star,a,\alpha,\beta>0$, such that $V(y)\geq \chi_d(y)$ when $|y|>a$, such that, for $n\geq n_0$, and for any $y\in\text{Supp}(\mathscr L(y_n))$
\[\mathcal L_nV(y)\leq-\alpha V(y)+\beta.\]
\end{enumerate}
\end{assumption}

In this assumption, the function $V$ is a so-called Lyapunov function. The multi-index $d$ can be thought of as $d=d_1$ (which is sufficient for Theorem~\ref{theorem:APTrPT} to hold). However, in the setting of Assumption~\ref{assumption:variance}, it might be necessary to consider $d>d_1$. Of course, if Assumption~\ref{assumption:moments} holds for $d'>d$, then it holds for $d$. Note that we usually can take $V(y)=\e^{\theta y}$, so that we can choose every component of $d$ as large as needed.

\begin{remark}[\textit{ii)} $\Rightarrow$ \textit{i)}]
Computing $\E[\chi_d(y_n)]$ to check Assumption~\ref{assumption:moments}.i) can be involved, so we rather check a Lyapunov criterion. It is classic that \textit{ii)} entails \textit{i)}. Indeed, denoting by $n_1:=n_0\vee\min\{n\in\N^\star:\gamma_n<\alpha^{-1}\}$ and $v_n:=\E[V(y_n)]$, it is clear that
\[v_{n+1}\leq v_n+\gamma_{n+1}(\beta-\alpha v_n).\]
From this inequality, it is easy to deduce that, for $n\geq n_1$, $v_{n+1}\leq\beta\alpha^{-1}\vee v_n$ and then by induction $v_n\leq\beta\alpha^{-1}\vee v_{n_1}$, which entails \textit{i)}. Then,
\begin{align*}
\E[\chi_d(y_n)]&=\prob(|y_n|\leq a)\E[\chi_d(y_n)||y_n|\leq a]+ \prob(|y_n|>a)\E[\chi_d(y_n)||y_n|>a]\\
&\leq\chi_d(a)+\frac\beta\alpha\vee \left(\sup_{k\leq n_1} v_k\right).
\end{align*}
\end{remark}

Note that, with a classical approach, Assumption~\ref{assumption:moments} would provide tightness and Assumption~\ref{assumption:convergence} would be used to identify the limit.

The previous three assumptions are crucial to provide a result on asymptotic pseudotrajectories (Theorem~\ref{theorem:APTrPT}), but are not enough to quantify speeds of convergence. As it can be observed in the proof of Theorem~\ref{theorem:APTrPT}, such speed relies deeply on the asymptotic behavior of $\gamma_{m(t)}$ and $\epsilon_{m(t)}$. To this end, we follow the guidelines of \cite{Ben99} to provide a condition in order to ensure such an exponential decay. For any non-increasing sequences $(\gamma_n),(\epsilon_n)$ converging to 0, define
\[\lambda(\gamma,\epsilon)=-\limsup_{n\to\infty}\frac{\log(\gamma_n\vee\epsilon_n)}{\sum_{k=1}^n\gamma_k},\]
where $\gamma$ and $\epsilon$ respectively stand for the sequences $(\gamma_n)_{n\geq 0}$ and $(\epsilon_n)_{n\geq 0}$.

\begin{remark}[Computation of $\lambda(\gamma,\epsilon)$]
With the notation of \cite[Proposition~8.3]{Ben99}, we have $\lambda(\gamma,\gamma)=-l(\gamma)$. It is easy to check that, if $\epsilon_n\leq\gamma_n$ for $n$ large, $\lambda(\gamma,\epsilon)=\lambda(\gamma,\gamma)$ and, if $\epsilon_n=\gamma_n^\beta$ with $\beta\leq1$, $\lambda(\gamma,\epsilon)=\beta\lambda(\gamma,\gamma)$. We can mimic \cite[Remark~8.4]{Ben99} to provide sufficient conditions for $\lambda(\gamma,\epsilon)$ to be positive. Indeed, if $\gamma_n=f(n),\epsilon_n=g(n)$ with $f,g$ two positive functions decreasing toward 0 such that $\int_1^{+\infty}f(s)ds=+\infty$, then
\[\lambda(\gamma,\epsilon)=-\limsup_{x\to\infty}\frac{\log\left(f(x)\vee g(x)\right)}{\int_1^xf(s)ds}.\]
Typically, if
\[\gamma_n\sim\frac{A}{n^a\log(n)^b},\quad\epsilon_n\sim\frac{B}{n^c\log(n)^d}\]
for $A,B,a,b,c,d\geq0$, then
\begin{itemize}
	\item $\lambda(\gamma,\epsilon)=0$ for $a<1$.
	\item $\lambda(\gamma,\epsilon)=(c\wedge1)A^{-1}$ for $a=1$ and $b=0$.
	\item $\lambda(\gamma,\epsilon)=+\infty$ for $a=1$ and $0<b\leq1$.
\end{itemize}
\label{remark:LambdaGammaEpsilon}
\end{remark}

Now, let us provide the main results of this paper.

\begin{theorem}[Asymptotic pseudotrajectories]
Let $(y_n)_{n\geq0}$ be an inhomogeneous Markov chain and let $\Phi$ and $\mu$ be defined as in \eqref{eq:DefPhit} and \eqref{eq:DefMut}. If Assumptions~\ref{assumption:convergence}, \ref{assumption:regularity}, \ref{assumption:moments} hold, then $(\mu_t)_{t\geq0}$ is an asymptotic pseudotrajectory of $\Phi$ with respect to $d_{\mathscr F}$, where
\[\mathscr F=\left\{f\in\mathcal D(\mathcal L)\cap\mathscr C^{N_2}_b:\mathcal Lf\in\mathcal D(\mathcal L),\|\mathcal Lf\|_\infty+\|\mathcal L\mathcal Lf\|_\infty+\sum_{j=0}^{N_2}\|  f^{(j)} \|_\infty\leq1\right\}.\]
Moreover, if $\lambda(\gamma,\epsilon)>0$, then $(\mu_t)_{t\geq0}$ is a $\lambda(\gamma,\epsilon)$-pseudotrajectory of $\Phi$ with respect to $d_{\mathscr F}$.
\label{theorem:APTrPT}\end{theorem}

\subsection{Consequences}
\label{subsection:consequences}

Theorem~\ref{theorem:APTrPT} relates the asymptotic behavior of the Markov chain $(y_n)$ to the one of the Markov process generated by $\mathcal L$. However, to deduce convergence or speeds of convergence of the Markov chain, we need another assumption: 

\begin{assumption}[Ergodicity]	
Assume that there exist a probability distribution $\pi$, constants $v,M_3>0$ ($M_3$ depending on $\mathscr L(y_0)$), and a class of functions $\mathscr G$ such that one of the following conditions holds:
\begin{enumerate}[i)]
	\item $\mathscr G\subseteq\mathscr F$ and, for any probability measure $\nu$, for all $t>0$,
	\[d_{\mathscr G}(\Phi(\nu,t),\pi)\leq d_{\mathscr G}(\nu,\pi)M_3\e^{-vt}.\]
	\item There exists $r,M_4>0$ such that, for all $s,t>0$
	\[d_{\mathscr G}(\Phi(\mu_s,t),\pi)\leq M_3\e^{-vt},\]
	and, for all $T>0$, with $C_T$ defined in Assumption~\ref{assumption:regularity},
	\[TC_T\leq M_4\e^{rT}.\]
	\item There exist functions $\psi: \mathbb{R}_+ \to \mathbb{R}_+ $ and $W\in\mathscr C^0$ such that
	\[	\lim_{t \to \infty} \psi(t) = 0, \quad \lim_{\|x\|\to \infty} W(x) = + \infty, \quad \sup_{n\geq0} \E[W(y_n)] < \infty,\]
	and, for any probability measure $\nu$, for all $t\geq 0$,
	\[d_{\mathscr G}(\Phi(\nu,t),\pi)\leq \nu(W) \psi(t).\]
	
\end{enumerate}
\label{assumption:ergodicity}
\end{assumption}

Since standard proofs of geometric ergodicity rely on the use of Grönwall's Lemma, Assumption~\ref{assumption:ergodicity}.i) and ii) are quite classic. In particular, using Foster-Lyapunov methods entails such inequalities (see e.g. \cite{MT93b,HM11}). However, in a weaker setting (sub-geometric ergodicity for instance) Assumption~\ref{assumption:ergodicity}.iii) might still hold; see for example \cite[Theorem~3.6]{JR02}, \cite[Theorem~3.2]{DFG09} or \cite[Theorem~4.1]{Hai10}. Note that, if $W=\chi_d$, then $\sup_{n\geq0} \E[W(y_n)] < \infty$ automatically from Assumption~\ref{assumption:moments}. Note that, in classical settings where $TC_T\leq M_4\e^{rT}$, we have \emph{i)$\Rightarrow$ ii)$\Rightarrow$ iii)}.

\begin{theorem}[Speed of convergence toward equilibrium]
Assume that Assumptions~\ref{assumption:convergence}, \ref{assumption:regularity}, \ref{assumption:moments} hold and let $\mathscr F$ be as in Theorem~\ref{theorem:APTrPT}.
\begin{enumerate}[i)]
	\item If Assumption~\ref{assumption:ergodicity}.i) holds and $\lambda(\gamma,\epsilon)>0$ then, for any $u<\lambda(\gamma,\epsilon)\wedge v$, there exists a constant $M_5$ such that, for all $t>t_0:=(v-u)^{-1}\log(1\wedge M_3)$,
\[d_{\mathscr{G}}\left(\mu_t,\pi\right)\leq\left(M_5+d_{\mathscr{G}}\left(\mu_0,\pi\right)\right)\e^{-ut}.\]
	\item If Assumption~\ref{assumption:ergodicity}.ii) holds and $\lambda(\gamma,\epsilon)>0$ then, for any $u<v\lambda(\gamma,\epsilon)(r+v+\lambda(\gamma,\epsilon))^{-1}$, there exists a constant $M_5$ such that, for all $t>0$,
\[d_{\mathscr F\cap\mathscr{G}}\left(\mu_t,\pi\right)\leq M_5\e^{-ut}.\]
	\item If Assumption~\ref{assumption:ergodicity}.iii) holds and convergence in $d_{\mathscr G}$ implies weak convergence, then $\mu_t$ converges weakly toward $\pi$ when $t\to\infty$.
\end{enumerate}
\label{theorem:SpeedConv}\end{theorem}

The first part of this theorem is similar to \cite[Lemma~8.7]{Ben99} but provides sharp bounds for the constants. In particular, $M_5$ and $t_0$ do not depend on $\mu_0$ (in Theorem~\ref{theorem:SpeedConv}.\emph{i)} only), see the proof for an explicit expression of $M_5$). The second part, however, does not require $\mathscr G$ to be a subset of $\mathscr F$, which can be rather involved to check, given the expression of $\mathscr F$ given in Theorem~\ref{theorem:APTrPT}. The third part is a direct consequence of \cite[Theorem~6.10]{Ben99}; we did not meet this case in our main examples, but we discuss the convergence toward sub-geometrically ergodic limit processes in Remark~\ref{remark:slowmixing}.

\begin{remark}[Rate of convergence in the initial scale]
Theorem~\ref{theorem:SpeedConv}.i) and ii) provide a bound of the form
$$
d_\mathscr{H} (\mathcal{L}(Y_t), \pi) \leq C e^{-u t},
$$
for some $\mathscr{H},C,u$ and all $t\geq0$. This easily entails, for another constant $C$ and all $n\geq 0$,
$$
d_\mathscr{H} (\mathcal{L}(y_n), \pi) \leq C e^{-u \tau_n}.
$$
Let us detail this bound for three examples where $\epsilon \leq \gamma$:\begin{itemize}
	\item if $\gamma_n= A n^{-1/2}$, then $d_\mathscr{H} (\mathcal{L}(y_n), \pi) \leq C e^{-2 A u\sqrt{n}}$.
	\item if $\gamma_n= A n^{-1}$, then $d_\mathscr{H} (\mathcal{L}(y_n), \pi) \leq C n^{-Au}$.
	\item if $\gamma_n= A(n\log(n))^{-1}$, then $d_\mathscr{H} (\mathcal{L}(y_n), \pi) \leq C \log(n)^{-Au}$.
\end{itemize}
In a nutshell, if $\gamma_n$ is large, the speed of convergence is good but $\lambda(\gamma,\gamma)$ is small.  In particular, even if $\gamma_n=n^{-1/2}$ provides the better speed, Theorem~\ref{theorem:SpeedConv} does not apply. Remark that the parameter $u$ is more important at the discrete time scale than it is at the continuous time scale.
\end{remark}

\begin{remark}[Convergence of unbounded functionals]
Theorem~\ref{theorem:SpeedConv} provides convergence in distribution of $(\mu_t)$ toward $\pi$, i.e. for every $f\in\mathscr C_b^0(\R^D)$,
\[\lim_{t\to\infty}\mu_t(f)=\pi(f).\]
Nonetheless, Assumption~\ref{assumption:moments} enables us to extend this convergence to unbounded functionals $f$. Recall that, if a sequence $(X_n)_{n\geq 0}$ converges weakly to $X$ and
\[M:=\E[V(X)]+\sup_{n\geq 0}\E[V(X_n)]<+\infty\]
for some positive function $V$, then $\E[f(X_n)]$ converges to $\E[f(X)]$ for every function $|f|<V^\theta$, with $\theta<1$. Indeed, let $(\kappa_m)_{m\geq0}$ be a sequence of $\mathscr C^\infty_c$ functions such that $\forall x\in\R^D,\lim_{m\to\infty}\kappa_m(x)=1$ and $0\leq\kappa_m\leq1$. We have, for $m\in\N$,
\begin{align*}
|\E\left[f(X_n)-f(X)\right]|&\leq|\E\left[(1-\kappa_m(X_n))f(X_n)\right]|+|\E\left[(1-\kappa_m(X))f(X)\right]|\\
&\quad+|\E\left[f(X_n)\kappa_m(X_n)-f(X)\kappa_m(X)\right]|\\
&\leq \E[|f(X_n)|^{\frac1\theta}]^\theta\E[(1-\kappa_m(X_n))^{\frac1{1-\theta}}]^{1-\theta}\\
&\quad+\E[|f(X)|^{\frac1\theta}]^\theta\E[(1-\kappa_m(X))^{\frac1{1-\theta}}]^{1-\theta}\\
&\quad+|\E\left[f(X_n)\kappa_m(X_n)-f(X)\kappa_m(X)\right]|\\
&\leq M^\theta\E[(1-\kappa_m(X_n))^{\frac1{1-\theta}}]^{1-\theta}+M^\theta\E[(1-\kappa_m(X))^{\frac1{1-\theta}}]^{1-\theta}\\
&\quad+|\E\left[f(X_n)\kappa_m(X_n)-f(X)\kappa_m(X)\right]|,
\end{align*}
so that, for all $m\in\N$,
\[\limsup_{n\to\infty}\E\left[f(X_n)-f(X)\right]\leq2M^\theta\E[(1-\kappa_m(X))^{\frac1{1-\theta}}]^{1-\theta}.\]
Using the dominated convergence theorem, $\lim_{n\to\infty}\E\left[f(X_n)-f(X)\right]=0$ since the right-hand side converges to 0. Note that the condition $|f|\leq V^\theta$ can be slightly weakened using the generalized Hölder's inequality on Orlicz spaces (see e.g. \cite{CDLA12}). Although, note that $\E[V(X_n)]$ may not converge to $\E[V(X)].$
\end{remark}

The following assumption is purely technical but is easy to verify in all of our examples, and will be used to prove functional convergence.

\begin{assumption}[Control of the variance]
Define the following operator:
\[\Gamma_nf=\mathcal L_nf^2-\gamma_{n+1}(\mathcal L_nf)^2-2f\mathcal L_nf.\]
Assume that there exists a multi-index $d_2$ and $M_6>0$ such that, if $\varphi_i$ is the projection on the $i^\text{th}$ coordinate,
\[\mathcal{L}_n \varphi_i(y) \leq M_6 \chi_{d_2} (y), \quad \Gamma_n \varphi_i(y) \leq M_6 \chi_{d_2} (y),\]
and
\[\mathcal{L}_n \chi_{d_2}(y) \leq M_6 \chi_{d_2} (y), \quad \Gamma_n \chi_{d_2}(y) \leq M_6 \chi_{d} (y),\]
where $d$ is defined in Assumption~\ref{assumption:moments}.
\label{assumption:variance}
\end{assumption}

\begin{theorem}[Functional convergence]
Assume that Assumptions~\ref{assumption:convergence}, \ref{assumption:regularity}, \ref{assumption:moments}, \ref{assumption:ergodicity} hold and let $\pi$ be as in Assumption~\ref{assumption:ergodicity}. Let $Y^{(t)}_s:=Y_{t+s}$ and $X^\pi$ be the process generated by $\mathcal L$ such that $\mathscr L(X^\pi_0)=\pi$. Then, for any $m\in\N^\star$, let $0<s_1<\dots<s_m$,
\[(Y^{(t)}_{s_1},\dots,Y^{(t)}_{s_m})\overset{\mathscr L}{\longrightarrow}(X^\pi_{s_1},\dots,X^\pi_{s_m}).\]

Moreover, if Assumption~\ref{assumption:variance} holds, then the sequence of processes $(Y^{(t)}_s)_{s\geq0}$ converges in distribution, as $t\to+\infty$, toward $(X^\pi_s)_{s\geq0}$ in the Skorokhod space.
\label{theorem:functionalCV}
\end{theorem}

For reminders about the Skorokhod space, the reader may consult \cite{JM86,Bil99,JS03}. Note that the operator $\Gamma_n$ we introduced in Assumption~\ref{assumption:variance} is very similar to the \emph{carré du champ} operator in the  continuous-time case, up to a term $\gamma_{n+1}(\mathcal L_nf)^2$ vanishing as $n\to+\infty$ (see e.g. \cite{Bak94,ABC+00,JS03}). Moreover, if we denote by $(K_n)$ the transition kernels of the Markov chain $(y_n)$, then it is clear that
\[\forall n\in\N,\quad \gamma_{n+1}\Gamma_nf=K_nf^2-(K_nf)^2.\]

\section{Illustrations}\label{section:illustrations}

\subsection{Weighted Random Walks}\label{subsection:randomwalk}
In this section, we apply Theorems~\ref{theorem:APTrPT}, \ref{theorem:SpeedConv} and \ref{theorem:functionalCV} to weighted random walks (WRWs) on $\R^D$. Let $(\omega_n)$ be a positive sequence, and $\gamma_n:=\omega_n(\sum_{k=1}^n\omega_k)^{-1}$. Then, set
\[x_n:=\frac{\sum_{k=1}^n\omega_kE_k}{\sum_{k=1}^n\omega_k},\quad x_{n+1}:=x_n+\gamma_{n+1}\left(E_{n+1}-x_n\right).\]
Here, $x_n$ is the weighted mean of $E_1,\dots,E_n$, where $(E_n)$ is a sequence of centered independent random variables. Under standard assumptions on the moments of $E_n$, the strong law of large numbers holds and $(x_n)$ converges to 0 a.s. Thus, it is natural to apply the general setting of Section~\ref{section:results} to $y_n:=x_n\gamma_n^{-1/2}$ and to define $\mu_t$ as in \eqref{eq:DefMut}. As we shall see, computations lead to the convergence of $\mathcal L_n$, as defined in \eqref{eq:DefLn}, toward
\[\mathcal Lf(y):=-ylf'(y)+\frac{\sigma^2}2f''(y),\]
where $l$ and $\sigma$ are defined below. Hence, the properly normalized process asymptotically behaves like the Ornstein-Uhlenbeck process; see Figure~\ref{figure_RandomWalk}. This process is the solution of the following stochastic differential equation (SDE):
\[dX_t=-lX_tdt+\sigma dW_t,\]
see \cite{Bak94} for instance. In the sequel, define $\mathscr F$ as in Theorem~\ref{theorem:APTrPT} with $N_2=3$, and $\varphi_i$ the projection on the $i^\text{th}$ coordinate.

\begin{proposition}[Results for the WRW]
Assume that
\[\E\left[\sum_{i=1}^D\varphi_i(E_{n+1})^2\right]=\sigma^2,\quad\sup_{n\geq 1}\gamma_n^2\omega_n^4\E[\|E_n\|^4]<+\infty,\quad\sup_n\gamma_n \sum_{i=1}^n\omega_i^2<+\infty,\]
and that there exist $l>0$ and $\beta>1$ such that
\begin{equation}
\sqrt{\frac{\gamma_n}{\gamma_{n+1}}}-1-\sqrt{\gamma_n\gamma_{n+1}}=-\gamma_nl+\mathcal O(\gamma_n^{\beta}).
\label{eq:RW1}\end{equation}
Then $(\mu_t)$ is an asymptotic pseudotrajectory of $\Phi$, with respect to $d_{\mathscr F}$.

Moreover, if $\lambda(\gamma,\gamma^{(\beta-1)\wedge\frac12})>0$ then, for any $u<l\lambda(\gamma,\gamma^{(\beta-1)\wedge\frac12})(l+\lambda(\gamma,\gamma^{(\beta-1)\wedge\frac12}))^{-1}$, there exists a constant $C$ such that, for all $t>0$,
\begin{equation}
d_{\mathscr F}\left(\mu_t,\pi\right)\leq C\e^{-ut},
\label{eq:RW3}
\end{equation}
where $\pi$ is the Gaussian distribution $\mathscr N\left(0,\sigma^2/(2l)\right)$.

Moreover, the sequence of processes $(Y^{(t)}_s)_{s\geq0}$ converges in distribution, as $t\to+\infty$, toward $(X^\pi_s)_{s\geq0}$ in the Skorokhod space.
\label{proposition:ResultWRW}\end{proposition}

\begin{figure}[htbp!]
\includegraphics[width=\textwidth]{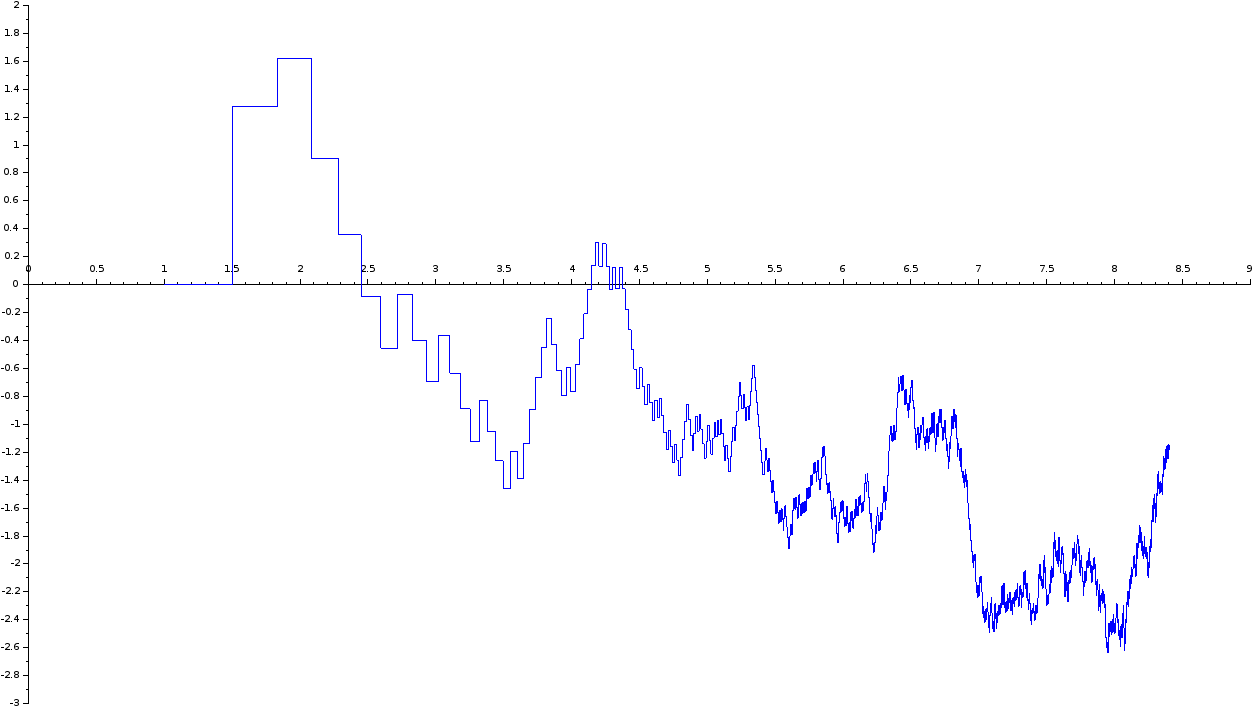}
\caption{Trajectory of the interpolated process for the normalized mean of the WRW with $\omega_n=1$ and $\mathscr L(E_n)=(\delta_{-1}+\delta_1)/2$.}
\label{figure_RandomWalk}\end{figure}

It is possible to recover the functional convergence using classical results: for instance, one can apply \cite[Theorem~2.1, Chapter~10]{KY03} with a slightly stronger assumption on $(\gamma_n)$. Yet, to our knowledge, the rate of convergence \eqref{eq:RW3} is original.

\begin{remark}[Powers of $n$]
Typically, if $\gamma_n\sim An^{-\alpha}$, then we can easily check that
\begin{itemize}
	\item if $\alpha=1$, then \eqref{eq:RW1} holds with $l=1-\frac1{2A}$ and $\beta=2$.
	\item if $0<\alpha<1$, then \eqref{eq:RW1} holds with $l=1$ and $\beta=\frac{1+\alpha}\alpha>2$.
\end{itemize}
Observe that, if $\omega_n=n^a$ for any $a>-1$, then $\gamma_n\sim\frac{1+a}n$ and \eqref{eq:RW1} holds with $l=\frac{1+2a}{2+2a}$ and $\beta=2$.
\label{remark:gammarandomwalk}
\end{remark}

We will see during the proof that checking Assumptions~\ref{assumption:convergence}, \ref{assumption:regularity}, \ref{assumption:moments} and \ref{assumption:ergodicity} is quite direct.

\begin{proof}[Proof of Proposition~\ref{proposition:ResultWRW}]
For the sake of simplicity, we do the computations for $D=1$. We have
\[y_{n+1}=\sqrt{\frac{\gamma_n}{\gamma_{n+1}}}y_n+\sqrt{\gamma_{n+1}}(E_{n+1}-\sqrt{\gamma_n}y_n),\]
so
\[\mathcal L_nf(y)=\gamma_{n+1}^{-1}\E[f(y_{n+1})-f(y_n)|y_n=y]=\gamma_{n+1}^{-1}\E[f(y+I_n(y))-f(y)],\]
with $I_n(y):=\left(\sqrt{\frac{\gamma_n}{\gamma_{n+1}}}-1-\sqrt{\gamma_n\gamma_{n+1}}\right)y+{\sqrt{\gamma_{n+1}}}E_{n+1}$. Simple Taylor expansions provide the following equalities (where $\mathcal O$ is the Landau notation, deterministic and uniform over $y$ and $f$, and $\beta:=\beta\wedge\frac32$):
\begin{align*}
I_n(y)&=\left(-\gamma_nl+\mathcal O(\gamma_n^{\beta})\right)y+\sqrt{\gamma_{n+1}}E_{n+1},\\
I_n^2(y)&=\gamma_{n+1}E^2_{n+1}+\chi_2(y)(1+E_{n+1})\mathcal O\left(\gamma_{n+1}^{\beta}\right),\\
I_n^3(y)&=\chi_3(y)(1+E_{n+1}+E_{n+1}^2+E_{n+1}^3)\mathcal O\left(\gamma_{n+1}^{\beta}\right).
\end{align*}
In the setting of Remark~\ref{remark:gammarandomwalk}, note that $\beta=\frac32$. Now, Taylor formula provides a random variable $\xi_n^y$ such that
\[f(y+I_n(y))-f(y)=I_n(y)f'(y)+\frac{I_n^2(y)}2f''(y)+\frac{I_n^3(y)}6f^{(3)}(\xi_n^y).\]
Then, it follows that
\begin{align}
\mathcal L_nf(y)&=\gamma_{n+1}^{-1}\E\left[\left.I_n(y)f'(y)+\frac{I_n^2(y)}2f''(y)+\frac{I_n^3(y)}6f^{(3)}(\xi_n^y)\right|y_n=y\right]\notag\\
 &=\gamma_{n+1}^{-1}\left[\left(-\gamma_nl+\mathcal O(\gamma_n^{3/2})\right)y+\sqrt{\gamma_{n+1}}\E[E_{n+1}]\right]f'(y)\notag\\
 &\quad+\frac1{2\gamma_{n+1}}\left[\gamma_{n+1}\E[E^2_{n+1}+\chi_2(y)\mathcal O\left(\gamma_{n+1}^{\beta}\right)\right]f''(y)\notag\\
 &\quad+\gamma_{n+1}^{-1}\chi_3(y)\E[1+E_{n+1}+E_{n+1}^{2}+E_{n+1}^{3}]\|f^{(3)}\|_\infty\mathcal O\left(\gamma_{n+1}^{\beta}\right)\notag\\
 &=-ylf'(y)+\chi_1(y)\|f'\|_\infty\mathcal O\left(\gamma_n^{\beta-1}\right)+\frac{\sigma^2}2f''(y)+\chi_2(y)\|f''\|_\infty\mathcal O\left(\gamma_n^{\beta-1}\right)\notag\\
 &\quad+\chi_3(y)\|f^{(3)}\|_\infty\mathcal O\left(\gamma_n^{\beta-1}\right).\label{eq:RW2}
\end{align}
From \eqref{eq:RW2}, we can conclude that
\[\left|\mathcal L_nf(y)-\mathcal Lf(y)\right|=\chi_3(y)(\|f'\|_\infty+\|f''\|_\infty+\|f^{(3)}\|_\infty)\mathcal O(\gamma_n^{\beta-1}).\]
As a consequence, the WRW satisfies Assumptions~\ref{assumption:convergence} with $d_1=3$, $N_1=3$ and $\epsilon_n=\gamma_n^{\beta-1}$. Note that (see Remark~\ref{remark:LambdaGammaEpsilon}) $\lambda(\gamma,\epsilon)=\beta-1$ if $\gamma_n=n^{-1}$.

Now, let us show that $P_tf$ admits bounded derivatives for $f\in\mathscr F$. Here, the expressions of the semigroup and its derivatives are explicit and the computations are simple (see \cite{Bak94,ABC+00}). Indeed, $P_tf(x)=\E[f(x\e^{-lt}+\sqrt{1-\e^{-2lt}}G)]$ and $(P_tf)^{(j)}(y)=\e^{-jlt}P_tf^{(j)}(y)$, where $\mathscr L(G)=\mathscr N(0,1)$. Then, it is clear that
\[\|(P_tf)^{(j)}\|_\infty=\e^{-jlt}\|P_tf^{(j)}\|_\infty\leq\|f^{(j)}\|_\infty.\]
Hence Assumption~\ref{assumption:regularity} holds with $N_2=3$ and $C_T=1$. Without loss of generality (in order to use Theorem~\ref{theorem:functionalCV} later) we set $d=4$.

Now, we check that the moments of order 4 of $y_n$ are uniformly bounded. Applying Cauchy-Schwarz's inequality:
\[\E\left[\left\|\sum_{i=1}^n\omega_iE_i\right\|^4\right]=\E\left[\sum_{i=1}^n\omega_i^4\|E_i\|^4+6\sum_{i<j}\omega_i^2 \|E_i^2\|\omega_j^2 \|E_j\|^2\right]\leq C\left(\sum_{i=1}^n\omega_i^2\right)^2,\]
for some explicit constant $C$. Then, since
\[\E[\|y_n\|^4]=\gamma_n^2\E\left[\left\|\sum_{i=1}^n\omega_iE_i\right\|^4\right]\leq C\sup_{n\geq1}\left(\gamma_n\sum_{i=1}^n\omega_i^2\right)^2,\]
the sequence $(y_n)_{n\geq0}$ satisfies Assumption~\ref{assumption:moments}.

It is classic, using coupling methods with the same Brownian motion for instance, that, for any probability measure $\nu$,
\[d_{\mathscr G}(\Phi(\nu,t),\pi)\leq d_{\mathscr G}(\nu,\pi)\e^{-lt},\]
where $\pi=\mathscr N\left(0,\sigma^2/(2l)I_D\right)$ and $d_\mathscr{G}$ is the Wasserstein distance ($\mathscr G$ is the set of 1-Lipschitz functions, see \cite{Che04}). We have, for $s,t>0$,
\[d_{\mathscr G}(\Phi(\mu_s,t),\pi)\leq d_{\mathscr G}(\mu_s,\pi)\e^{-lt}\leq(M_2+\pi(\chi_1))\e^{-lt}.\]
In other words, Assumption~\ref{assumption:ergodicity}.ii) holds for the WRW model with $M_3=M_2+\pi(\chi_1),M_4=1,v=l,r=0$ and $\mathscr F\subseteq\mathscr G$.

Finally, it is easy to check Assumption~\ref{assumption:variance} in the case of the WRW, with $d_2=2$, and then $\Gamma_n\chi_2\leq M_6\chi_4$ (that is why we set $d=4$ above).

Then, Theorems~\ref{theorem:APTrPT}, \ref{theorem:SpeedConv} and \ref{theorem:functionalCV} achieve the proof of Proposition~\ref{proposition:ResultWRW}.
\end{proof}

\begin{remark}[Building a limit process with jumps]
In this paper, we mainly provide examples of Markov chains converging (in the sense of Theorem~\ref{theorem:APTrPT}) toward diffusion processes (see Section~\ref{subsection:randomwalk}) or jump processes (see Section~\ref{subsection:bandit}). However, it is not hard to adapt the previous model to obtain an exemple converging toward a diffusion process with jumps (see Figure~\ref{figure_diffusionjump}): this illustrates how every component (drift, jump and noise) appears in the limit generator. The intuition is that the jump terms appear when larger and larger jumps of the Markov chain occur with smaller and smaller probability. For an example when $D=1$, take
\[\omega_n:=1,\quad E_n:=
\left\{\begin{array}{ll}
F_n&\text{ if }U_n\geq\sqrt{\gamma_n}\\
\gamma_n^{-1/2}G_n&\text{ if }U_n<\sqrt{\gamma_n}
\end{array}\right.
,\quad y_n:=\frac1{\sqrt{\gamma_n}}\sum_{k=1}^nE_k,\]
where $(F_n)_{n\geq1},(G_n)_{n\geq1}$ and $(U_n)_{n\geq1}$ are three sequences of i.i.d. random variables, such that $\E[F_1]=0,\E[F_1^2]=\sigma^2,\mathscr L(G_1)=Q$, $\mathscr L(U_1)$ is the uniform distribution on $[0,1]$. In this case, $\gamma_n=1/n$ and it is easy to show that $\mathcal L_n$ as defined in \eqref{eq:DefLn} converges toward the following infinitesimal generator:
\[\mathcal Lf(y):=-\frac12yf'(y)+\frac{\sigma^2}2f''(y)+\int_\R[f(y+z)-f(y)]Q(dz),\]
so that Assumption~\ref{assumption:convergence} holds with $d_1=3$, $N_1=3,\epsilon_n=n^{-1/2}$.

\begin{figure}[htbp!]
\includegraphics[width=\textwidth]{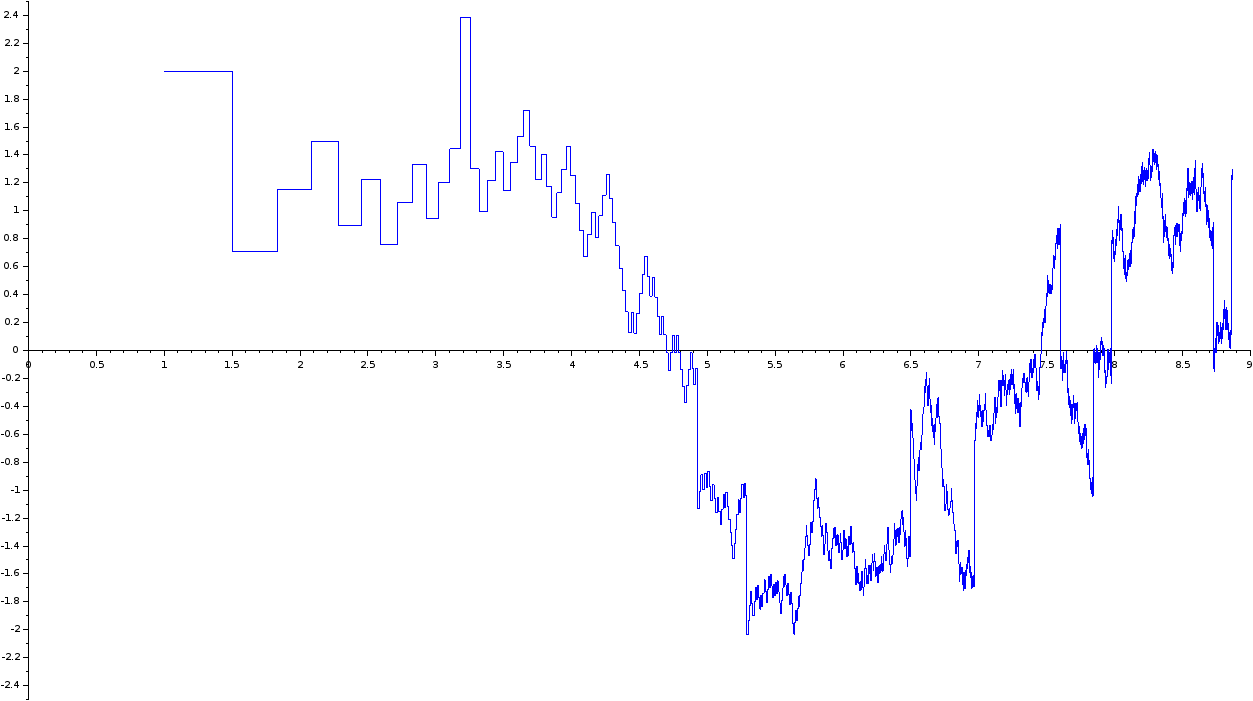}
\caption{Trajectory of the interpolated process for the toy model of Remark~\ref{remark:RWjumps} with $\mathscr L(F_n)=\mathscr L(G_n)=(\delta_{-1}+\delta_1)/2$.}
\label{figure_diffusionjump}\end{figure}

\label{remark:RWjumps}\end{remark}

\subsection{Penalized Bandit Algorithm}
\label{subsection:bandit}
In this section, we slightly generalize the penalized bandit algorithm (PBA) model introduced by Lamberton and Pagès, and we recover \cite[Theorem~4]{LP08}. Such algorithms aim at optimizing the gain in a game with two choices, $A$ and $B$, with respective unknown gain probabilities $p_A$ and $p_B$. Originally, $A$ and $B$ are the two arms of a slot machine, or \emph{bandit}. Throughout this section, we assume $0\leq p_B<p_A\leq 1$.

Let $s:[0,1]\to[0,1]$ be a function, which can be understood as a player's strategy, such that $s(0)=0,s(1)=1$. Let $x_n\in[0,1]$ be a measure of her trust level in $A$ at time $n$. She chooses $A$ with probability $s(x_n)$ independently from the past, and updates $x_n$ as follows:
\begin{center}\begin{tabular}{|l|l|l|}
\hline
$x_{n+1}$&Choice&Result\\\hline
$x_n+\gamma_{n+1}(1-x_n)$&$A$&Gain\\\hline
$x_n-\gamma_{n+1}x_n$&$B$&Gain\\\hline
$x_n+\gamma_{n+1}^2(1-x_n)$&$B$&Loss\\\hline
$x_n-\gamma_{n+1}^2x_n$&$A$&Loss\\\hline
\end{tabular}\end{center}
Then $(x_n)$ satisfies the following Stochastic Approximation algorithm:
\[x_{n+1}:=x_n+\gamma_{n+1}\left(X_{n+1}-x_n\right)+\gamma_{n+1}^2\left(\widetilde X_{n+1}-x_n\right),\]
where
\begin{equation}
(X_{n+1},\widetilde X_{n+1}):=
\left\{\begin{array}{ll}
(1,x_n)&\text{ with probability }p_1(x_n)\\
(0,x_n)&\text{ with probability }p_0(x_n)\\
(x_n,1)&\text{ with probability }\widetilde p_1(x_n)\\
(x_n,0)&\text{ with probability }\widetilde p_0(x_n)
\end{array}\right.,
\label{eq:PBP11}
\end{equation}
with
\begin{equation}
p_1(x)=s(x)p_A,\quad p_0(x)=(1-s(x))p_B,\quad \widetilde p_1(x)=(1-s(x))(1-p_B),\quad \widetilde p_0(x)=s(x)(1-p_A).
\label{eq:PBP12}
\end{equation}
Note that the PBA of \cite{LP08} is recovered by setting $s(x)=x$ in \eqref{eq:PBP12}.

From now on, we consider the algorithm \eqref{eq:PBP11} where $p_1,p_0,\widetilde p_1,\widetilde p_0$ are non-necessarily given by \eqref{eq:PBP12}, but are general non-negative functions whose sum is 1. Let $\mathscr F$ be as in Theorem~\ref{theorem:APTrPT} with $N_2=2$, and $y_n:=\gamma_n^{-1}(1-x_n)$ the rescaled algorithm. Let $\mathcal L_n$ be defined as in \eqref{eq:DefLn},
\begin{equation}
\mathcal Lf(y):=[\widetilde p_0(1)-yp_1(1)]f'(y)-yp_0'(1)[f(y+1)-f(y)],
\label{eq:PBP6}
\end{equation}
and $\pi$ the invariant distribution for $\mathcal L$ (which exists and is unique, see Remark~\ref{remark:stationaryPDMP}).

Under the assumptions of Proposition~\ref{proposition:ResultPBA}, it is straightforward to mimic the results \cite{LP08} and ensure that our generalized algorithm $(x_n)_{n\geq0}$ satisfies the ODE Lemma (see e.g. \cite[Theorem~2.1, Chapter~5]{KY03}), and converges toward 1 almost surely.

\begin{proposition}[Results for the PBA]
Assume that $\gamma_n=n^{-1/2}$, that $p_1,\widetilde p_1,\widetilde p_0\in C^1_b, p_0\in C^2_b$, and that
\[p_0(1)=\widetilde p_1(1)=0,\quad p_0'(1)\leq0,\quad p_1(1)+p_0'(1)>0,\quad\widetilde p_1(0)>0.\]
If, for $0<x<1$, $(1-x)p_1(x)>xp_0(x)$, then $(\mu_t)$ is an asymptotic pseudotrajectory of $\Phi$, with respect to $d_{\mathscr F}$.

Moreover, $(\mu_t)$ converges to $\pi$ and the sequence of processes $(Y^{(t)}_s)_{s\geq0}$ converges in distribution, as $t\to+\infty$, toward $(X^\pi_s)_{s\geq0}$ in the Skorokhod space.
\label{proposition:ResultPBA}\end{proposition}

\begin{figure}[htbp!]
\includegraphics[width=\textwidth]{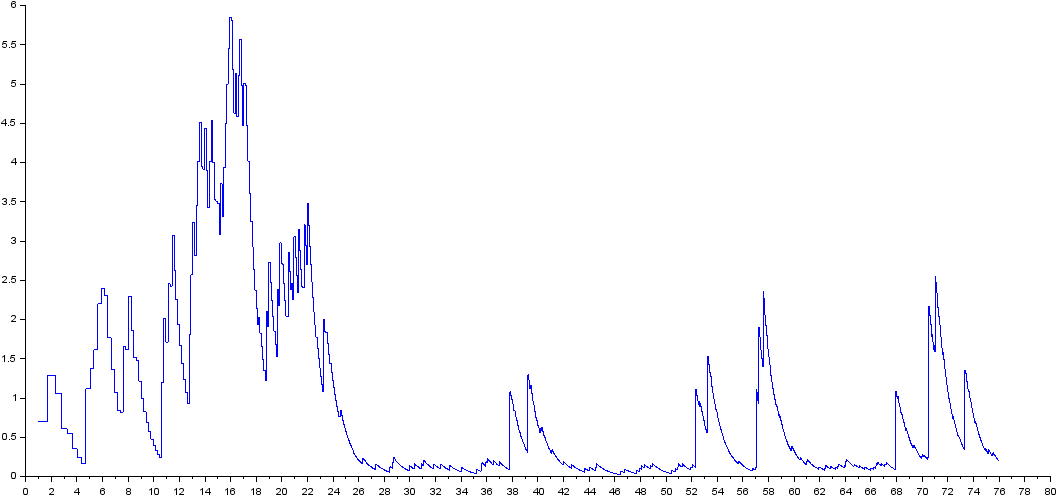}
\caption{Trajectory of the interpolated process for the rescaled PBA, setting $s(x)=x$ in \eqref{eq:PBP12}.}
\label{figure_PenalizedBandit}\end{figure}

The proof is given at the end of the section; before that, let us give some interpretation and heuristic explanation of the algorithm. The random sequence $(y_n)$ satisfies
\[y_{n+1}=y_n+\left(\frac{\gamma_n}{\gamma_{n+1}}-1\right)y_n-(X_{n+1}-x_n)-\gamma_{n+1}(\widetilde X_{n+1}-x_n),\]
thus, defining $\mathcal L_n$ as in \eqref{eq:DefLn},
\[\mathcal L_nf(y)=\gamma_{n+1}^{-1}\E\left[\left.f(y_{n+1})-f(y_n)\right|y_n=y\right]=\gamma_{n+1}^{-1}\E\left[f(y+I_n(y))-f(y)|y_n=y\right],\]
where
\begin{equation}
I_n(y):=
\left\{\begin{array}{ll}
I^1_n(y):=\left(\frac{\gamma_n}{\gamma_{n+1}}-1-\gamma_n\right) y &\text{ with probability }p_1(1 - \gamma_n y)\\
I^0_n(y):=1+\left(\frac{\gamma_n}{\gamma_{n+1}}-1-\gamma_n\right)  y &\text{ with probability }p_0(1 - \gamma_n y)\\
\widetilde I^1_n(y):=\left(\frac{\gamma_n}{\gamma_{n+1}}-1-\gamma_n\gamma_{n+1}\right)y&\text{ with probability }\widetilde p_1(1 - \gamma_n y)\\
\widetilde I^0_n(y):=\gamma_{n+1}+\left(\frac{\gamma_n}{\gamma_{n+1}}-1-\gamma_n\gamma_{n+1}\right)y &\text{ with probability }\widetilde p_0(1 - \gamma_n y)
\end{array}\right..
\label{eq:PBP7}
\end{equation}
Taylor expansions provide the convergence of $\mathcal L_n$ toward $\mathcal L$. As a consequence, the properly renormalized interpolated process will asymptotically behave like a PDMP (see Figure~\ref{figure_PenalizedBandit}). Classically, one can read the dynamics of the limit process through its generator (see e.g. \cite{Dav93}): the PDMP generated by \eqref{eq:PBP6} has upward jumps of height 1 and follows the flow given by the ODE $y'=\widetilde p_0(1)-yp_1(1)$, which means it converges exponentially fast toward~$\widetilde p_0(1)/p_1(1)$.

\begin{remark}[Interpretation]
Consider the case \eqref{eq:PBP12}. Here Proposition~\ref{proposition:ResultPBA} states that the rescaled algorithm $(y_n)$ behaves asymptotically like the process generated by
\[\mathcal Lf(x)=(1-p_A-xp_A)f'(x)+p_Bs'(1)x[f(x+1)-f(x)].\]
Intuitively, it is more and more likely to play the arm $A$ (the one with the greatest gain probability). Its successes and failures appear within the drift term of the limit infinitesimal generator, whereas playing the arm $B$ with success will provoke a jump. Finally, playing the arm $B$ with failure does not affect the limit dynamics of the process (as $\widetilde p_1$ does not appear within the limit generator). To carry out the computations in this section, where we establish the speed of convergence of $(\mathcal L_n)$ toward $\mathcal L$, the main idea is to condition $\E[y_{n+1}]$ given typical events on the one hand, and rare events on the other hand. Typical events generally construct the drift term of $\mathcal L$ and rare events are responsible of the jump term of $\mathcal L$ (see also Remark~\ref{remark:RWjumps}).

Note that one can tune the frequency of jumps with the parameter $s'(1)$. The more concave $s$ is in a neighborhood of 1, the better the convergence is. In particular, if $s'(1)=0$, the limit process is deterministic. Also, note that choosing a function $s$ non-symmetric with respect to $(1/2,1/2)$ introduces an a priori bias; see Figure~\ref{figure_strategies}.

\begin{figure}[htbp!]
\begin{center}
\begin{tikzpicture}[scale=3]
\draw[dashed] (0,1) -- (1,1);
\draw[dashed] (1,0) -- (1,1);
\draw[->] (-0.1,0) node[below]{0} -- (1,0) node[below]{1};
\draw[->] (0,-0.1) -- (0,1) node[left]{1};
\draw [domain=0:1,thick,smooth,color=blue] plot (\x,{1*\x});
\end{tikzpicture}
\qquad
\begin{tikzpicture}[scale=3]
\draw[dashed] (0,1) -- (1,1);
\draw[dashed] (1,0) -- (1,1);
\draw[->] (-0.1,0) node[below]{0} -- (1,0) node[below]{1};
\draw[->] (0,-0.1) -- (0,1) node[left]{1};
\draw [domain=0:0.5,thick,smooth,color=blue] plot (\x,{4*\x^3});
\draw [domain=0.5:1,thick,smooth,color=blue] plot (\x,{4*(\x-1)^3+1});
\end{tikzpicture}
\qquad
\begin{tikzpicture}[scale=3]
\draw[dashed] (0,1) -- (1,1);
\draw[dashed] (1,0) -- (1,1);
\draw[->] (-0.1,0) node[below]{0} -- (1,0) node[below]{1};
\draw[->] (0,-0.1) -- (0,1) node[left]{1};
\draw [domain=0:1,thick,smooth,color=blue] plot (\x,{-(\x-1)^2+1});
\end{tikzpicture}
\end{center}
\caption{Various strategies for $s(x)=x$, $s$ concave, $s$ with a bias}
\label{figure_strategies}
\end{figure}
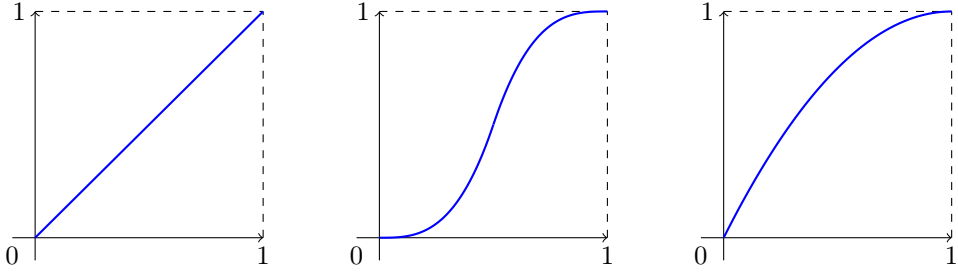

\label{remark:stategyPBA}
\end{remark}

Let us start the analysis of the rescaled PBA with a global result about a large class of PDMPs, whose proof is postponed to Section~\ref{section:appendix}. This lemma provides the necessary arguments to check Assumption~\ref{assumption:regularity}.

\begin{proposition}[Assumption~\ref{assumption:regularity} for PDMPs]
Let $X$ be a PDMP with infinitesimal generator
\[\mathcal Lf(x)=(a-bx)f'(x)+(c+dx)[f(x+1)-f(x)],\]
such that $a,b,c,d\geq0$. Assume that either $b>0$, or $b=0$ and $a\neq0$. If $f\in\mathscr C^N_b$, then, for all $0\leq t\leq T$, $P_tf\in\mathscr C_b^N$. Moreover, for all $n\leq N$,
\[\|(P_tf)^{(n)}\|_\infty\leq
\left\{ \begin{array}{ll}
\sum_{k=0}^n{\left(\frac{2|d|}b\right)^{n-k}\|f^{(k)}\|_\infty} &\text{if }b>0\\
\sum_{k=0}^n\frac{n!}{k!}(2|d|T)^{n-k}\|f^{(k)}\|_\infty&\text{if }b=0
\end{array} \right..\]
\label{proposition:regPDMP}\end{proposition}

Note that a very similar result is obtained in \cite{BR15b}, but for PDMPs with a diffusive component.

\begin{remark}[The stationary probability distribution]
Let $(X_t)_{t\geq0}$ be the PDMP generated by $\mathcal L$ defined in Proposition~\ref{proposition:regPDMP}. By using the same tools as in \cite[Theorem~6]{LP08}, it is possible to prove existence and uniqueness of a stationary distribution $\pi$ on $\R_+$. Applying Dynkin's formula with $f(x)=x$, we get
\[\partial_t\E[X_t]=a+c-(b-d)\E[X_t].\]
If one uses the same technique with $f(x)=x^n$, it is possible to deduce the $n^\text{th}$ moment of the invariant measure $\pi$, and Dynkin's formula applied to $f(x)=\exp(\lambda x)$ provides exponential moments of $\pi$ (see \cite[Remark~2.2]{BMPPR15} for the detail).

In the setting of \eqref{eq:PBP6}, one can use the reasoning above to show that, by denoting by $m_n=\int_0^\infty x^n\pi(dx)$ for $n\geq0$,
\[m_n=\frac{-p_0'(1)}{n(p_1(1)+p_0'(1))}\sum_{k=1}^{n-2}\binom{n}{k-1}m_k+\frac{2\widetilde p_0(1)+(n-1)p_0'(1)}{2(p_1(1)+p_0'(1))}m_{n-1},\]
with the convention $\sum_{k=i+1}^i=0$.
\label{remark:stationaryPDMP}
\end{remark}

\begin{proof}[Proof of Proposition~\ref{proposition:ResultPBA}]
First, let us specify the announced convergence of $\mathcal L_n$ toward $\mathcal L$; recall that $\gamma_n=n^{-1/2}$ and $\chi_d(y)=\sum_{k=0}^d|y|^k$, so that $I_n(y)$ in \eqref{eq:PBP7} rewrites
\[ I_n(y)=
\left\{\begin{array}{ll}
\frac{\sqrt{n+1} - \sqrt{n} - 1}{\sqrt{n}} y &\text{ with probability }p_1(1 - \gamma_n y)\\
1+\frac{\sqrt{n+1} - \sqrt{n} - 1}{\sqrt{n}}  y &\text{ with probability }p_0(1 - \gamma_n y)\\
\frac{\sqrt{n}- \sqrt{n+1}}{\sqrt{n+1}} y&\text{ with probability }\widetilde p_1(1 - \gamma_n y)\\
\frac1{\sqrt{n+1}}+\frac{\sqrt{n}- \sqrt{n+1}}{\sqrt{n+1}} y &\text{ with probability }\widetilde p_0(1 - \gamma_n y)
\end{array}\right.,\]
and the infinitesimal generator rewrites
\begin{align}
\mathcal L_nf(y)&=\frac{p_1(1-\gamma_ny)}{\gamma_{n+1}}\left[f\left(y+I_n^1(y)\right)-f(y)\right]+\frac{p_0(1-\gamma_ny)}{\gamma_{n+1}}\left[f\left(y+I_n^0(y)\right)-f(y)\right]\notag\\
&\quad+\frac{\widetilde p_1(1-\gamma_ny)}{\gamma_{n+1}}\left[f\left(y+\widetilde I_n^1(y)\right)-f(y)\right]+\frac{\widetilde p_0(1-\gamma_ny)}{\gamma_{n+1}}\left[f\left(y+\widetilde I_n^0(y)\right)-f(y)\right].\label{eq:PBP8}
\end{align}
In the sequel, the Landau notation $\mathcal O$ will be deterministic and uniform over both $y$ and $f$.

First, we consider the first term of \eqref{eq:PBP8} and observe that
\[p_1(1-\gamma_ny)=p_1(1)+y\mathcal O(\gamma_n),\]
and that
\[I_n^1(y)=\left(\frac{\gamma_n}{\gamma_{n+1}}-1-\gamma_n\right)y=\left(\frac1{2n}+o(n^{-1})-\frac1{\sqrt n}\right)y=-y\gamma_n(1+\mathcal O(\gamma_n)),\]
so that $I_n^1(y)^2=y^2\mathcal O(\gamma_n^2).$ Since $\gamma_n\sim\gamma_{n+1}$, and since the Taylor formula gives a random variable $\xi_n^y$ such that
\[f\left(y+I_n^1(y)\right)-f(y)=I_n^1(y)f'(y)+\frac{I_n^1(y)^2}2f''(\xi_n^y),\]
we have
\[\gamma_{n+1}^{-1}\left[f\left(y+I_n^1(y)\right)-f(y)\right]=-yf'(y)+\chi_2(y)(\|f'\|_\infty+\|f''\|_\infty)\mathcal O(\gamma_n).\]
Then, easy computations show that
\begin{equation}
\frac{p_1(1-\gamma_ny)}{\gamma_{n+1}}\left[f\left(y+I_n^1(y)\right)-f(y)\right]=-p_1(1)yf'(y)+\chi_3(y)(\|f'\|_\infty+\|f''\|_\infty)\mathcal O(\gamma_n).
\label{eq:PBP1}
\end{equation}
The third term in \eqref{eq:PBP8} is expanded similarly and writes
\begin{equation}
\frac{\widetilde p_1(1-\gamma_ny)}{\gamma_{n+1}}\left[f\left(y+\widetilde I_n^1(y)\right)-f(y)\right]=\chi_3(y)(\|f'\|_\infty+\|f''\|_\infty)\mathcal O(\gamma_n),
\label{eq:PBP2}
\end{equation}
while the fourth term becomes
\begin{equation}
\frac{\widetilde p_0(1-\gamma_ny)}{\gamma_{n+1}}\left[f\left(y+\widetilde I_n^0(y)\right)-f(y)\right]=\widetilde p_0(1)f'(y)+\chi_3(y)(\|f'\|_\infty+\|f''\|_\infty)\mathcal O(\gamma_n).
\label{eq:PBP3}
\end{equation}
Note the slight difference with the expansion of the second term, since we have, on the one hand,
\[\frac{p_0(1-\gamma_ny)}{\gamma_{n+1}}=-\frac{\gamma_n}{\gamma_{n+1}} yp_0'(1)+\frac{\gamma_n^2}{\gamma_{n+1}}y^2p''(\xi_n^y) =-yp_0'(1)+\chi_2(y)\mathcal O(\gamma_n),\]
where $\xi_n^y$ is a random variable, while, on the other hand,
\[f(y+I_n^0(y))-f(y)=f(y+1)-f(y)+\chi_1(y)\|f'\|_\infty\mathcal O(\gamma_n).\]
Then,
\begin{align}
&\frac{p_0(1-\gamma_ny)}{\gamma_{n+1}}\left[f\left(y+I_n^0(y)\right)-f(y)\right]=\notag\\
&\quad-yp_0'(1)[f(y+1)-f(y)]+\chi_3(y)(\|f\|_\infty+\|f'\|_\infty)\mathcal O(\gamma_n).
\label{eq:PBP4}
\end{align}
Finally, combining \eqref{eq:PBP1}, \eqref{eq:PBP2}, \eqref{eq:PBP3} and \eqref{eq:PBP4}, we obtain the following speed of convergence for the infinitesimal generators:
\begin{equation}
\left|\mathcal L_nf(y)-\mathcal Lf(y)\right|=\chi_3(y)(\|f\|_\infty+\|f'\|_\infty+\|f''\|_\infty)\mathcal O(\gamma_n),
\label{eq:PBP5}
\end{equation}
establishing that the rescaled PBA satisfies Assumption~\ref{assumption:convergence} with $d_1=3$, $N_1=2$ and $\epsilon_n=\gamma_n$. Assumption~\ref{assumption:regularity} follows from Proposition~\ref{proposition:regPDMP} with $N_2=2$.

In order to apply Theorem~\ref{theorem:APTrPT}, it would remain to check Assumption~\ref{assumption:moments}, that is to prove that the moments of order 3 of $(y_n)$ are uniformly bounded. This happens to be very difficult and we do not even know whether it is true. As an illustration of this difficulty, the reader may refer to \cite[Remark~4.4]{GPS15}, where uniform bounds for the first moment are provided using rather technical lemmas, and only for an overpenalized version of the algorithm.

In order to overcome this technical difficulty, we introduce a truncated Markov chain coupled with $(y_n)$, which does satisfy a Lyapunov criterion. For $l\in\N^\star$ and $\delta\in(0,1]$, we define $(y_n^{(l,\delta)})_{n\geq0}$ as follows:
\[y_n^{(l,\delta)}:=
\left\{\begin{array}{ll}
y_n&\text{for }n\leq l\\
\left(y_{n-1}^{(l,\delta)}+I_{n-1}(y_{n-1}^{(l,\delta)})\right)\wedge\delta\gamma_n^{-1}&\text{for }n>l
\end{array}\right..\]
In the sequel, we denote with an exposant ${(l,\delta)}$ the equivalents of $\mathcal L_n,Y_t,\mu_t$ for $(y_n^{(l,\delta)})_{n\geq0}$. We prove that $(\mathcal L_n^{(l,\delta)})_{n\geq0}$ satisfies our main assumptions, and consequently $(\mu_t^{(l,\delta)})_{t\geq0}$ is an asymptotic pseudotrajectory of $\Phi$ (at least for $\delta$ small enough and $l$ large enough), which is the result of the combination of Lemma~\ref{lemma:truncationAPT} and Theorem~\ref{theorem:APTrPT}.

\begin{lemma}[Behavior of $(\mu_t^{(l,\delta)})_{t\geq0}$]
For $\delta$ small enough and $l$ large enough, the inhomogeneous Markov chain $(y_n^{(l,\delta)})_{n\geq0}$ satisfies Assumptions~\ref{assumption:convergence}, \ref{assumption:regularity}, \ref{assumption:moments} and \ref{assumption:variance}.
\label{lemma:truncationAPT}
\end{lemma}

Now, we shall prove that $(\mu_t)_{t\geq0}$ is an asymptotic pseudotrajectory of $\Phi$ as well. Indeed, let $\varepsilon>0$ and $l$ be large enough such that $\prob(\forall n\geq l,\gamma_ny_n\leq\delta)\geq1-\varepsilon$ (it is possible since $\gamma_ny_n=1-x_n$ converges to 0 in probability). Then, for $T>0,f\in\mathscr F,s\in[0,T]$
\begin{align*}
\left|\mu_{t+s}(f)-\Phi(\mu_t,s)(f)\right|&\leq \left|\mu_{t+s}(f)-\mu^{(l,\delta)}_{t+s}(f)\right|+\left|\Phi(\mu^{(l,\delta)}_t,s)(f)-\Phi(\mu_t,s)(f)\right|\\
	&\quad+\left|\mu^{(l,\delta)}_{t+s}(f)-\Phi(\mu^{(l,\delta)}_t,s)(f)\right|\\
&\leq (2\|f\|_\infty+2\|f\|_\infty)(1-\prob(\forall n\geq l,\gamma_ny_n\leq\delta))\\
&\quad+\left|\mu^{(l,\delta)}_{t+s}(f)-\Phi(\mu^{(l,\delta)}_t,s)(f)\right|\\
&\leq4\varepsilon+\left|\mu^{(l,\delta)}_{t+s}(f)-\Phi(\mu^{(l,\delta)}_t,s)(f)\right|,
\end{align*}
since $\|f\|_\infty\leq1$. Taking the suprema over $[0,T]$ and $\mathscr F$ yields
\begin{equation}
\limsup_{t\to\infty}\sup_{s\in[0,T]}d_{\mathscr F}(\mu_{t+s},\Phi(\mu_t,s))\leq4\varepsilon+\limsup_{t\to\infty}\sup_{s\in[0,T]}d_{\mathscr F}(\mu^{(l,\delta)}_{t+s},\Phi(\mu^{(l,\delta)}_t,s)).
\label{eq:PBP9}
\end{equation}
Using Lemma~\ref{lemma:truncationAPT}, Theorem~\ref{theorem:APTrPT} holds for $(\mu_t^{(l,\delta)})_{t\geq0}$ and \eqref{eq:PBP9} rewrites
\[\limsup_{t\to\infty}\sup_{s\in[0,T]}d_{\mathscr F}(\mu_{t+s},\Phi(\mu_t,s))\leq4\varepsilon,\]
so that $(\mu_t)_{t\geq0}$ is an asymptotic pseudotrajectory of $\Phi$.

Finally, for $t>0,T>0,f\in\mathscr C^0_b,s\in[0,T]$, set $\nu_t:=\mathscr L((Y^{(t)}_s)_{0\leq T})$ and $\nu:=\mathscr L((X^\pi_s)_{0\leq T})$. We have
\begin{align}
\left|\nu_t(f)-\nu(f)\right|&\leq \left|\nu_t(f)-\nu^{(l,\delta)}_t(f)\right|+\left|\nu^{(l,\delta)}_t(f)-\nu(f)\right|\notag\\
&\leq 2\|f\|_\infty(1-\prob(\forall n\geq l,\gamma_ny_n\leq\delta))	+\left|\nu^{(l,\delta)}_t(f)-\nu(f)\right|\notag\\
&\leq2\varepsilon+\left|\nu^{(l,\delta)}_t(f)-\nu(f)\right|.
\label{eq:PBP10}
\end{align}
Since $(y_n^{(l,\delta)})_{n\geq0}$ satisfies Assumption~\ref{assumption:variance}, we can apply Theorem~\ref{theorem:functionalCV} so that the right-hand side of \eqref{eq:PBP10} converges to 0, which concludes the proof.
\end{proof}

\begin{remark}[Rate of convergence toward the stationary measure]
For such PDMPs, exponential convergence in Wasserstein distance has already been obtained (see \cite[Proposition~2.1]{BMPPR15} or \cite[Theorem~3.4]{GPS15}). However, we are not in the setting of Theorem~\ref{theorem:SpeedConv}, since $\gamma_n=n^{-1/2}$. Thus, $\lambda(\gamma,\epsilon)=0$, and there is no exponential convergence. This highlights the fact that the rescaled algorithm converges too slowly toward the limit PDMP.
\end{remark}

\begin{remark}[The overpenalized bandit algorithm]
Even though we do not consider the overpenalized bandit algorithm introduced in \cite{GPS15}, the tools are the same. The behavior of this algorithm is the same as the PBA's, except from a possible (random) penalization of an arm in case of a success; it writes
\[x_{n+1}=x_n+\gamma_{n+1}\left(X_{n+1}-x_n\right)+\gamma_{n+1}^2\left(\widetilde X_{n+1}-x_n\right),\]
where
\[(X_{n+1},\widetilde X_{n+1})=
\left\{\begin{array}{ll}
(1,x_n)&\text{ with probability }p_Ax_n\sigma\\
(0,x_n)&\text{ with probability }p_B(1-x_n)\sigma\\
(1,0)&\text{ with probability }p_Ax_n(1-\sigma)\\
(0,1)&\text{ with probability }p_B(1-x_n)(1-\sigma)\\
(x_n,1)&\text{ with probability }(1-p_B)(1-x_n)\\
(x_n,0)&\text{ with probability }(1-p_A)x_n
\end{array}\right..\]
Setting $y_n=\gamma_n^{-1}(1-x_n)$, and following our previous computations, it is easy to show that the rescaled overpenalized algorithm converges, in the sense of Assumption~\ref{assumption:convergence}, toward
\[\mathcal Lf(y)=[1-\sigma p_A-p_Ay]f'(y)+p_By[f(y+1)-f(y)].\]
\end{remark}

\subsection{Decreasing Step Euler Scheme}\label{subsection:euler}
In this section, we turn to the study of the so-called decreasing step Euler scheme (DSES). This classical stochastic procedure is designed to approximate the stationary measure of a diffusion process of the form
\begin{equation}
X_t^x=x+\int_0^tb(X_s)ds+\int_0^t\sigma(X_s)dW_s
\label{eq:DefSDE}
\end{equation}
with a discrete Markov chain
\begin{equation}
y_{n+1}:=y_n+\gamma_{n+1}b(y_n)+\sqrt{\gamma_{n+1}}\sigma(y_n)E_{n+1},
\label{eq:DefDSES}
\end{equation}
for any non-increasing sequence $(\gamma_n)_{n\geq1}$ converging toward 0 such that $\sum_{n=1}^\infty\gamma_n=+\infty$ and $(E_n)$ a suitable sequence of random variables. In the sequel, we shall recover the convergence of the DSES toward the diffusion process at equilibrium, as defined by \eqref{eq:DefSDE}. If $\gamma_n=\gamma$ in \eqref{eq:DefDSES}, this model would be a constant step Euler scheme as studied by \cite{Tal84,TT90}, which approaches the diffusion process at time $t$ when $\gamma$ tends to 0. By letting $t\to+\infty$ in \eqref{eq:DefSDE}, it converges to the equilibrium of the diffusion process. We can concatenate those steps by choosing $\gamma_n$ vanishing but such that $\sum_n\gamma_n$ diverges. The DSES has already been studied in the literature, see for instance \cite{LP02,Lem05}.

It is simple, following the computations of Sections~\ref{subsection:randomwalk} and \ref{subsection:bandit}, to check that $\mathcal L_n$ converges (in the sense of Assumption~\ref{assumption:convergence}) toward
\[\mathcal Lf(y):=b(y)f'(y)+\frac{\sigma^2(y)}2f''(y).\]
In the sequel, define $\mathscr F$ as in Theorem~\ref{theorem:APTrPT} with $N_2=3$.

\begin{proposition}[Results for the DSES]
Assume that $(E_n)$ is a sequence of sub-gaussian random variables (i.e. there exists $\kappa>0$ such that $\forall\theta\in\R,\E[\exp(\theta E_1)]\leq\exp(\kappa\theta^2/2)$), and $\E[E_1]=0$ and $\E[E_1^2]=1$. Moreover, assume that $b,\sigma\in\mathscr C^\infty$ whose derivatives of any order are bounded, and that $\sigma$ is bounded. Eventually, assume that there exist constants $0<b_1\leq b_2$ and $0<\sigma_1$ such that, for $|y|>A$,
\begin{equation}
-b_2y^2\leq b(y)y\leq-b_1y^2,\quad\sigma_1\leq\sigma(y).
\label{eq:HypDSES}
\end{equation}
If $\gamma_n=1/n$, then $(\mu_t)$ is a $\frac12$-pseudotrajectory of $\Phi$, with respect to $d_{\mathscr F}$.

Moreover, there exists a probability distribution $\pi$ and $C,u>0$ such that, for all $t>0$,
\[d_{\mathscr F}\left(\mu_t,\pi\right)\leq C\e^{-ut}.\]

Furthermore, the sequence of processes $(Y^{(t)}_s)_{s\geq0}$ converges in distribution, as $t\to+\infty$, toward $(X^\pi_s)_{s\geq0}$ in the Skorokhod space.
\label{proposition:ResultDSES}
\end{proposition}

Note that one could choose a more general $(\gamma_n)$, provided that $\lambda(\gamma,\gamma)>0$. In contrast to classical results, Proposition~\ref{proposition:ResultDSES} provides functional convergence. Moreover, we obtain a rate of convergence in a more general setting than \cite[Theorem~IV.1]{Lem05}, see also \cite{LP02}. Indeed, let us detail the difference between those settings with the example of the Kolmogorov-Langevin equation: 
\[dX_t=\nabla V(X_t)dt +\sigma dB_t.\]
A rate of convergence may be obtained in \cite{Lem05} only for $V$ uniformly convex; although, we only need $V$ to be convex outside some compact set. Let us recall that the uniform convexity is a strong assumption ensuring log-Sobolev inequality, Wasserstein contraction\dots See for instance \cite{Bak94,ABC+00}.

\begin{proof}[Proof of Proposition~\ref{proposition:ResultDSES}]
Recalling $(y_n)$ in \eqref{eq:DefDSES} and $\mathcal L_n$ in \eqref{eq:DefLn}, we have
\[\mathcal L_n(y)=\gamma_{n+1}^{-1}\E\left[f(y+\gamma_{n+1}b(y)+\sqrt{\gamma_{n+1}}\sigma(y)E_{n+1})-f(y)|y_n=y\right].\]
Easy computations show that Assumption~\ref{assumption:convergence} holds with $\epsilon_n =\sqrt{\gamma_n},N_1=3,d_1=3$.

We aim at proving Assumption~\ref{assumption:regularity}, i.e. for $f\in\mathscr F,j\leq3$ and $t\leq T$, that $(P_tf)^{(j)}$ exists and
\[\|(P_tf)^{(j)}\|_\infty\leq C_T\sum_{k=0}^3\|f^{(k)}\|_\infty.\]
It is straightforward for $j=0$, but computations are more involved for $j\geq1$. Let us denote by $(X_t^x)_{t\geq0}$ the solution of \eqref{eq:DefSDE} starting at $x$. Since $b$ and $\sigma$ are smooth with bounded derivatives, it is standard that $x\mapsto X_t^x$ is $\mathscr C^4$ (see for instance \cite[Chapter~II, Theorem~3.3]{Kun82}). Moreover, $\partial_x X_t^x$ satisfies the following SDE:
\[\partial_x X_t^x=1+\int_0^tb'(X_s^x)\partial_xX_s^xds+ \int_0^t\sigma'(X_s^x)\partial_xX_s^xdW_s.\]
For our purpose, we need the following lemma, which provides a constant for Assumption~\ref{assumption:regularity} of the form $C_T=C_1\e^{C_2T}$. Even though we do not explicit the constants for the second and third derivatives in its proof, it is still possible; the main result of the lemma being that we can check Assumption~\ref{assumption:ergodicity}.ii).

\begin{lemma}[Estimates for the derivatives of the diffusion]
Under the assumptions of Proposition~\ref{proposition:ResultDSES}, for $p\geq2$ and $t\leq T$,
\[\E[|\partial_xX_t^x|^p]\leq\exp\left(\left(p\|b'\|_\infty+\frac{p(p-1)}2\|\sigma'\|_\infty^2\right)T\right),\]
and
\[\E[|\partial_xX_t^x|]\leq\exp\left(\left(\|b'\|_\infty+\frac12\|\sigma'\|_\infty^2\right)T\right).\]
For any $p\in\N^\star$, there exist positive constants $C_1,C_2$ not depending on $x$, such that
\[\E[|\partial_x^2X_t^x|^p]\leq C_1\e^{C_2T},\quad \E[|\partial_x^3X_t^x|^p]\leq C_1\e^{C_2T}.\]
\label{lemma:DerivativesDiffusion}
\end{lemma}
The proof of the lemma is postponed to Section~\ref{section:appendix}. Using Lemma~\ref{lemma:DerivativesDiffusion}, and since $f$ and its derivatives are bounded, it is clear that $x\mapsto P_tf(x)$ is three times differentiable, with
\begin{align*}
(P_tf)'(x)&=\E\Big[f'(X_t^x)\partial_xX_t^x\Big],\\
(P_tf)''(x)&=\E\Big[f''(X_t^x)(\partial_xX_t^x)^2+f'(X_t^x)(\partial^2_xX_t^x)\Big],\\
(P_tf)^{(3)}(x)&=\E\Big[f^{(3)}(X_t^x)(\partial_xX_t^x)^3+3f''(X_t^x)(\partial_xX_t^x)(\partial^2_xX_t^x)+f'(X_t^x)(\partial^3_xX_t^x)\Big].
\end{align*}
As a consequence, Assumption~\ref{assumption:regularity} holds, with $C_T=3C_1^3\e^{3C_2T}$ and $N_2=3$.

Now, we shall prove that Assumption~\ref{assumption:moments}.ii) holds with $V(y)=\exp(\theta y)$, for some (small) $\theta>0$. Thanks to \eqref{eq:HypDSES}, we easily check that, for $\widetilde V(y)=1+y^2$,
\begin{equation}
\mathcal L\widetilde V(y)\leq -\widetilde\alpha \widetilde V(y)+\widetilde\beta,\quad\text{with } \widetilde\alpha=2b_1,\widetilde\beta=(2b_1+S)\vee\left(A\sup_{[-A,A]}b+\frac{S^2}2+2b_1(1+A^2)\right).
\label{eq:LyapDSES}
\end{equation}
Then, \cite[Proposition~III.1]{Lem05} entails Assumption~\ref{assumption:moments}.ii). Finally, Theorem~\ref{theorem:APTrPT} applies and we recover \cite[Theorem~2.1, Chapter~10]{KY03}.

Then, Theorem~\ref{theorem:APTrPT} provides the asymptotic behavior of the Markov chain $(y_n)_{n\geq0}$ (in the sense of asymptotic pseudotrajectories). If furtherly we want speeds of convergence, we shall use Theorem~\ref{theorem:SpeedConv} and prove the ergodicity of the limit process; to that end, combine \eqref{eq:LyapDSES} with \cite[Theorem~6.1]{MT93b} (which provides exponential ergodicity for the diffusion toward some stationary measure $\pi$), as well as Lemma~\ref{lemma:DerivativesDiffusion}, to ensure Assumption~\ref{assumption:ergodicity}.ii) with $\mathscr G=\{g\in\mathscr C^0(\R):|g(y)|\leq1+y^2\}$ ($v$ and $r$ are not explicitly given). Note that we used the fact that $\sigma$ is lower-bounded, which implies that the compact sets are small sets. Moreover, the choice $\gamma_n=n^{-1}$ implies $\lambda(\gamma,\epsilon)=1/2$. Then, the assumptions of Theorem~\ref{theorem:SpeedConv} are satisfied, with $u_0=v(1+2v+2r)^{-1}$.

Finally, we can easily check Assumption~\ref{assumption:variance} for some $d\in\N$, since $y_n$ admits uniformly bounded exponential moments. Then using Theorem~\ref{theorem:functionalCV} ends the proof.
\end{proof}

\subsection{Lazier and Lazier Random Walk}

We consider the lazier and lazier random walk (LLRW) $(y_n)_{n\geq0}$ defined as follows:
\begin{equation}
y_{n+1}:=\left\{\begin{array}{ll}
y_n+Z_{n+1}&\text{with probability }\gamma_{n+1}\\
y_n&\text{with probability }1-\gamma_{n+1}
\end{array}\right.,
\label{eq:DefLLRW}
\end{equation}
where $(Z_n)$ is such that $\mathscr L(Z_{n+1}|y_0,\dots,y_n)=\mathscr L(Z_{n+1}|y_n)$; we denote the conditional distribution $Q(y_n,\cdot):=\mathscr L(Z_{n+1}|y_n)$. In the sequel, define $\mathscr F:=\left\{f\in\mathscr C^0_b:7\|f\|_\infty\leq1\right\}$ and $\mathcal Lf(y)=\int_\R f(y+z)Q(y,dz)-f(y)$, which is the generator of a pure-jump Markov process (constant between two jumps).

This example is very simple and could be studied without using our main results; however, we still develop it in order to check the sharpness of our rates of convergence (see Remak~\ref{remark:doeblin}).

\begin{proposition}[Results for the LLRW model]
The sequence $(\mu_t)$ is an asymptotic pseudotrajectory of $\Phi$, with respect to $d_{\mathscr F}$.

Moreover, if $\lambda(\gamma, \gamma)>0$, then $(\mu_t)$ is a $\lambda(\gamma, \gamma)$-pseudotrajectory of $\Phi$.

Furthermore, if $\mathcal L$ satisfies Assumption~\ref{assumption:ergodicity}.i) for some $v>0$ then, for any $u<v \wedge \lambda(\gamma, \gamma)$, there exists a constant $C$ such that, for all $t>0$,
\[d_{\mathscr F}\left(\mu_t,\pi\right)\leq C\e^{-ut}.\]
\label{proposition:ResultLLRW}\end{proposition}

Remark that the distance $d_{\mathscr F}$ in Proposition~\ref{proposition:ResultLLRW} is the total variation distance up to a constant.

\begin{proof}[Proof of Proposition~\ref{proposition:ResultLLRW}]
It is easy to check that \eqref{eq:DefLn} entails
\[\mathcal L_nf(y)=\int_\R f(y+z)Q(y,dz)-f(y)=\mathcal Lf(y).\]
It is clear that the LLRW satisfies Assumption~\ref{assumption:convergence} with $d_1=0, N_1=0,\epsilon_n=0$, and Assumption~\ref{assumption:regularity} with $C_T=1,N_2=0$. Since $d=d_1=0$, Assumption~\ref{assumption:moments} is also clearly satisfied. Eventually, note that $\lambda(\gamma,\epsilon)= \lambda(\gamma, \gamma)$. Then, Theorem~\ref{theorem:APTrPT} holds. Finally, if $\mathcal L$ satisfies Assumption~\ref{assumption:ergodicity}.i), it is clear that Theorem~\ref{theorem:SpeedConv} applies.
\end{proof}

The assumption on $\mathcal L$ satisfying Assumption~\ref{assumption:ergodicity}.i) (which strongly depends on the choice of $Q$), can be checked with the help of a Foster-Lyapunov criterion, see \cite{MT93b} for instance.

\begin{remark}[Constructing limit processes with a slow speed of convergence]
The framework of the LLRW provides a large pool of toy examples. Let $R$ be some Markov transition kernel on $\R$, and define $Q(y,A)=R(y,y+A),$ for any $y\in\R$ and $A$ borelian set, where $y+A=\{z\in\R:z-y\in A\}$. Let $(y_n)_{n\geq0}$ be the LLRW defined in \eqref{eq:DefLLRW}. Proposition~\ref{proposition:ResultLLRW} holds, and the limit process generated by $\mathcal Lf(y)=\int_\R f(y+z)Q(y,dz)-f(y)$ is just a Markov chain generated by $R$ indexed by a Poisson process. Precisely, if $N_t$ is a Poisson process of intensity 1,
\[\Phi(\nu,t)=\E[\nu R^{N_t}].\]

This construction allows us to build a variety of limit processes for the LLRW, with a slow speed of convergence if needed. Indeed, choose $R$ to be the Markov kernel of a sub-geometrically ergodic Markov chain converging to a stationary measure $\pi$ at polynomial speed (for instance the kernels introduced in \cite{JR02}); the limit process will inherit the slow speed of convergence. More precisely, there exist $\beta\geq1$, a class of functions $\mathscr G$ and a function $W$ such that
\[d_{\mathscr G}(\nu R^n,\pi)\leq \frac{\nu(W)}{(1+n)^\beta}.\]
Then,
\[d_{\mathscr G}(\Phi(\nu,t),\pi)\leq \E\left[\frac{\nu(W)}{(1+N_t)^\beta}\right]\]
which goes to 0 at polynomial speed. Then, if $\sup_n\E[W(y_n)]<+\infty,$ which could be proven via troncature arguments as in Section~\ref{subsection:bandit}, we can use Theorem~\ref{theorem:SpeedConv}.iii) to conclude that $(y_n)$ converges weakly toward $\pi$.

Note that another example of sub-geometrically ergodic process is provided in \cite[Theorem~5.4]{DFG09}. The elliptic diffusions mentionned in this article converge slowly toward equilibrium, and could be approximated by a Euler scheme as in Section~\ref{subsection:euler}. In this example again, the use of troncature arguments to check Assumption~\ref{assumption:moments} could be enough for Theorem~\ref{theorem:SpeedConv}.iii) to hold.
\label{remark:slowmixing}
\end{remark}

\begin{remark}[Speed of convergence under Doeblin condition]
Assume there exists a measure $\psi$ and $\varepsilon>0$ such that for every $y$ and measurable set $A$, we have
$$
\int \mathds{1}_{y +z \in A} Q(y,dz) \geq \varepsilon \psi(A).
$$
It is the classical Doeblin condition, which ensures exponential uniform ergodicity in total variation distance. It is classic to prove that under this condition there exists an invariant distribution $\pi$, such that , for every $\mu$ and $t\geq 0$
$$
d_{\mathscr{F}} ( \mu P_t ,\pi) \leq e^{-t \varepsilon} d_{\mathscr{F}} ( \mu ,\pi) \leq e^{-t \varepsilon}
$$
Indeed, one can couple two trajectories as follows: choose the same jump times and, using the Doeblin condition, at each jumps, couple them with probability $\varepsilon$. The coupling time then follows an exponential distribution with parameter $\varepsilon$. Then, the conclusion of Proposition~\ref{proposition:ResultLLRW} holds with $v = \varepsilon^{-1}$.

However, one can use the Doeblin argument directly with the inhomogeneous chain. Let us denote by $(K_n)$ its sequence of transition kernels.  From the Doeblin condition, we have for every $\mu, \nu$ and $n \geq 0$
$$
d_{\mathscr{F}} ( \mu K_n ,\nu K_n) \leq (1- \gamma_{n+1} \varepsilon) d_{\mathscr{F}} ( \mu ,\nu).
$$
and as $\pi$ is invariant for $K_n$ (it is straighforward because $\pi$ is invariant for $Q$) then
$$
d_{\mathscr{F}} ( \mu K_n ,\pi) \leq (1- \gamma_{n+1} \varepsilon) d_{\mathscr{F}} ( \mu ,\pi).
$$
A recursion argument then gives
$$
d_{\mathscr{F}} ( \mathcal{L}(y_n) ,\pi) \leq \prod_{k=0}^n (1- \gamma_{k+1} \varepsilon) d_{\mathscr{F}} ( \mathcal{L}(y_0) ,\pi ).
$$
But, 
$$
\prod_{k=0}^n (1- \gamma_{k+1} \varepsilon) = \exp\left( \sum_{k=0}^n \ln(1-\gamma_{k+1} \varepsilon) \right) \leq \exp\left( \sum_{k=0}^n \ln(1-\gamma_{k+1} \varepsilon) \right)\leq e^{- \varepsilon \sum_{k=0}^n \gamma_{k+1} }.
$$

As a conclusion, Proposition~\ref{proposition:ResultLLRW} and the direct approach provide the same rate of convergence for the LLRW under Doeblin condition.
\label{remark:doeblin}
\end{remark}

\begin{remark}[Non-convergence in total variation]
Assume that $y_n\in\mathbb{R}_+$ and $Z_n=  -y_n/2$. We then have that
\[y_n = \prod_{i=1}^n \widetilde{\Theta}_i y_0,\quad \widetilde{\Theta}_i=
\left\{\begin{array}{ll}
1&\text{with probability }1-\gamma_i\\
\frac12&\text{with probability }\gamma_i
\end{array}\right..\]
where $\widetilde{\Theta}_i$ are independent random variables. Borel-Cantelli's Lemma entails that $(y_n)_{n\geq0}$ converges to $0$ almost surely and, here,
$$
\mathcal{L} f(y) =  f\left(\frac{y}{2}\right) - f(y).
$$
A process with such a generator never hits $0$ whenever it starts with a positive value and, then, does not converge in total variation distance. Nevertheless, it is easy to prove that for any $y$ and $t\geq 0$,
$$
d_{\mathcal{G}}(\delta_y P_t, \delta_0) \leq \E\left[ \frac{1}{2^{N_t}} \right] y \leq e^{-t/2} y,
$$
where $\mathcal{G}$ is any class of functions included in $\{ f \in\mathscr C^1_b : \ \Vert f' \Vert _\infty \leq 1 \}$, and $(N_t)$ a Poisson process. In particular Assumption~\ref{assumption:ergodicity}.ii) holds and there is convergence of our chain to zero in distribution, as well as a rate of convergence in the Fortet-Mourier distance.
\end{remark}

\section{Proofs of theorems}\label{section:proofs}

In the sequel, we consider the following classes of functions:
\begin{align}
\mathscr F_1&:=\left\{f\in\mathcal{D}(\mathcal{L}):\mathcal Lf\in\mathcal{D}(\mathcal{L}), \|f\|_\infty + \| \mathcal L f \|_\infty + \| \mathcal L \mathcal L f \|_\infty \leq 1 \right\},\notag\\
\mathscr F_2&:= \left\{ f \in \mathcal{D}(\mathcal{L}) \cap \mathscr C^{N_2}_b  : \sum_{j=0}^{N_2}\|  f^{(j)} \|_\infty \leq 1 \right\},\notag\\
\mathscr F&:=\mathscr F_1\cap\mathscr F_2.\notag
\end{align}
The class $\mathscr F_1$ is particularly useful to control $P_tf$ (see Lemma~\ref{lemma:F1}), and the class $\mathscr F_2$ enables us to deal with smooth and bounded functions (for the second part of the proof of Theorem~\ref{theorem:APTrPT}). Note that an important feature of $\mathscr F$ is that Lemma~\ref{lemma:WeakConv} holds for $\mathscr F_1\cap\mathscr F_2$, so that $\mathscr F$ contains $\mathscr C^\infty_c$ "up to a constant".

Let us begin with preliminary remarks on the properties of the semigroup $(P_t)$.
\begin{lemma}[Expansion of $P_tf$]
Let $f\in\mathscr F_1$. Then, for all $t>0$, $P_tf\in\mathscr F_1$ and
\[\sup_{f\in\mathscr F_1}\|P_tf-f-t\mathcal Lf\|_\infty\leq\frac{t^2}2.\]
\label{lemma:F1}\end{lemma}

\begin{proof}[Proof of Lemma~\ref{lemma:F1}]
It is clear that $P_tf\in\mathscr F_1$, since for all $g\in\mathcal D(\mathcal L)$, $P_t\mathcal Lg=\mathcal LP_tg$ and $\|P_tg\|_\infty\leq\|g\|_\infty$. Now, if $f \in \mathscr{F}_1$, then
\[P_tf=f+\int_0^tP_s\mathcal{L}fds=f+t\mathcal{L}f+K(f,t),\]
where $K(f,t)=P_tf-f-t\mathcal Lf$. Using the mean value inequality, we have, for $x\in\R^D$,
\begin{align*}
|K(f,t)(x)|&= \left|\int_0^t P_s \mathcal{L} f(x) ds - \mathcal{L} f(x)\right| \leq \int_0^t |P_s \mathcal{L} f(x) - \mathcal{L} f(x)| ds\\
&\leq \int_0^t s \| \mathcal{L} \mathcal{L} f \|_\infty ds \leq\frac{t^2}2,
\end{align*}
which concludes the proof.
\end{proof}

\begin{proof}[Proof of Theorem~\ref{theorem:APTrPT}]
For every $t\geq 0$, set $K(f,t):=P_tf-f-t\mathcal Lf$ and recall that $m(t)=\sup\{n \geq 0 \ : \ t \geq \tau_n\}$. Then, we have $Y_{\tau_{m(t)}}=Y_t$ and $\tau_{m(t)}\leq t<\tau_{m(t)+1}$. Let $0<s<T$. Using the following telescoping sum, we have
\begin{align}
d_{\mathscr F}(\mu_{t+s},\Phi(\mu_t,s))
&=d_{\mathscr F}(\mu_{\tau_{m(t+s)}},\Phi(\mu_{\tau_{m(t)}},s))\notag\\
&\leq d_{\mathscr F}(\Phi(\mu_{\tau_{m(t)}}, \tau_{m(t+s)} - \tau_{m(t)}),\Phi(\mu_{\tau_{m(t)}},s))\notag\\
&\quad+ d_{\mathscr F}(\mu_{\tau_{m(t+s)}},\Phi(\mu_{\tau_{m(t)}}, \tau_{m(t+s)} - \tau_{m(t)}))\notag\\
&\leq d_{\mathscr F}(\Phi(\mu_{\tau_{m(t)}}, \tau_{m(t+s)} - \tau_{m(t)}),\Phi(\mu_{\tau_{m(t)}}, s))\notag\\
& \quad + \sum_{k= m(t)}^{m(t+s)-1} d_{\mathscr F}\left(\Phi\left(\mu_{\tau_{k+1}}, \sum_{j=k+2}^{m(t+s)} \gamma_j\right),\Phi\left(\mu_{\tau_{k}}, \sum_{j=k+1}^{m(t+s)} \gamma_j\right) \right),
\label{eq:APTProof4}\end{align}
with the convention $\sum_{k=i+1}^i=0$. Our aim is now to bound each term of this sum. The first one is the simplest: indeed, we have $s\leq\tau_{m(t+s)+1}-\tau_{m(t)}$, so $s-\gamma_{m(t+s)+1}\leq\tau_{m(t+s)}-\tau_{m(t)}$ and $\tau_{m(t+s)}-\tau_{m(t)}\leq s+\gamma_{m(t)+1}$. Denoting by $u=s\wedge(\tau_{m(t+s)}-\tau_{m(t)})$ and $h=|\tau_{m(t+s)}-\tau_{m(t)}-s|$ we have, by the semigroup property,
\[d_{\mathscr F}\left(\Phi(\mu_{t}, \tau_{m(t+s)} - \tau_{m(t)}),\Phi(\mu_{t}, s)\right)=d_{\mathscr F}\left(\Phi(\Phi(\mu_{t},u), h),\Phi(\mu_{t},u)\right).\]
From Lemma~\ref{lemma:F1}, we know that for every $ f\in\mathscr F_1$ and every probability measure $\nu$,
\[|\Phi(\nu,h)(f)-\nu(f)|=|\nu(P_hf-f)|\leq h + \frac{h^2}2\leq \frac32 h,\]
for $h\leq1$. It is then straightforward that 
\begin{equation}
d_{\mathscr F}\left(\Phi(\mu_{t}, \tau_{m(t+s)} - \tau_{m(t)}),\Phi(\mu_{t}, s)\right)\leq \frac32 h\leq\frac32 \gamma_{m(t)+1}.
\label{eq:APTProof1}\end{equation}
Now, we provide bounds for the generic term of the telescoping sum in \eqref{eq:APTProof4}. Let $f \in \mathscr F_1$ and $m(t)\leq k\leq m(t+s)-1$. On the one hand, using Lemma~\ref{lemma:F1},
\begin{align*}
\Phi\left(\mu_{\tau_k}, \sum_{j=k+1}^{m(t+s)} \gamma_j\right) (f)
&=\mu_{\tau_k} P_{\sum_{j=k+1}^{m(t+s)} \gamma_j} (f)\\
&=\mu_{\tau_k}(P_{\tau_{m(t+s)}-\tau_{k+1}} f) + \int_0^{ \gamma_{k+1}} \mu_{\tau_k} (\mathcal{L} P_{\tau_{m(t+s)}-\tau_{k+1}+u}f) du\\
&=\mu_{\tau_k}(P_{\tau_{m(t+s)}-\tau_{k+1}} f) + \gamma_{k+1} \mu_{\tau_k} (\mathcal{L} P_{\tau_{m(t+s)}-\tau_{k+1}}f)\\
&\quad+ K\left(P_{\tau_{m(t+s)}-\tau_{k+1}}f, \gamma_{k+1}\right).
\end{align*}
On the other hand,
\begin{align*}
\mu_{\tau_{k+1}} (f)
&= \mu_{\tau_{k}} (f) + \gamma_{k+1} \mu_{\tau_{k}} (\mathcal{L}_k f)
\end{align*}
so that
\begin{align*}
\Phi\left(\mu_{\tau_{k+1}}, \sum_{j=k+2}^{m(t+s)} \gamma_j\right) (f)
&=\mu_{\tau_{k+1}} (P_{\tau_{m(t+s)}-\tau_{k+1}} f)\\
&=\mu_{\tau_{k}} (P_{\tau_{m(t+s)}-\tau_{k+1}}f) + \gamma_{k+1} \mu_{\tau_{k}} (\mathcal{L}_k P_{\tau_{m(t+s)}-\tau_{k+1}} f).
\end{align*}
Henceforth,
\begin{align*}
\Phi\left(\mu_{\tau_{k+1}}, \sum_{j=k+2}^{m(t+s)} \gamma_j\right) (f) -\Phi\left(\mu_{\tau_k}, \sum_{j=k+1}^{m(t+s)} \gamma_j\right) (f)
&\leq \gamma_{k+1} \mu_{\tau_{k}} ((\mathcal{L}_k  - \mathcal{L})P_{\tau_{m(t+s)}-\tau_{k+1}} f)\\
&\quad+ K\left(P_{\tau_{m(t+s)}-\tau_{k+1}}f, \gamma_{k+1}\right).
\end{align*}
Now, we bound the previous term using Assumption~\ref{assumption:convergence}, Assumption~\ref{assumption:regularity}, and Assumption~\ref{assumption:moments}. Let $m(t)\leq k\leq m(t+s)-1$. Recall that, since $s<T$,  $\tau_{m(t+s)}-\tau_{k+1}\leq\tau_{m(t+s)}-\tau_{m(t)+1}\leq(t+s)-t\leq T$. Then, for all $f\in\mathscr F_2$,
\begin{align*}
&|\mu_{\tau_{k}} ((\mathcal{L}_k  - \mathcal{L})P_{\tau_{m(t+s)}-\tau_{k+1}} f)|\leq\mu_{\tau_{k}} (|(\mathcal{L}_k -\mathcal L)P_{\tau_{m(t+s)}-\tau_{k+1}} f|)\\
	&\quad\leq\mu_{\tau_{k}}\left(M_1\chi_{d_1}\sum_{j=0}^{N_1}\|(P_{\tau_{m(t+s)}-\tau_{k+1}} f)^{(j)}\|_\infty\epsilon_k\right)\leq\mu_{\tau_{k}}\left(M_1(N_1+1)C_T\chi_d\sum_{j=0}^{N_2}\|f^{(j)}\|_\infty\epsilon_k\right)\\
	&\quad\leq M_1(N_1+1)C_T\E[\chi_d(y_k)]\sum_{j=0}^{N_2}\|f^{(j)}\|_\infty\epsilon_k\leq M_1M_2(N_1+1)C_T\sum_{j=0}^{N_2}\|f^{(j)}\|_\infty\epsilon_k\\
	&\quad\leq M_1M_2(N_1+1)C_T\epsilon_k.
\end{align*}
Gathering the previous bounds entails
\begin{align}
&\sum_{k= m(t)}^{m(t+s)-1} d_{\mathscr F}\left(\Phi\left(\mu_{\tau_{k+1}}, \sum_{j=k+2}^{m(t+s)} \gamma_j\right),\Phi\left(\mu_{\tau_{k}}, \sum_{j=k+1}^{m(t+s)} \gamma_j\right) \right)\notag\\
&\quad\leq \sum_{k= m(t)}^{m(t+s)-1} \left(M_1M_2(N_1+1)C_T\gamma_{k+1}\epsilon_k+\frac{\gamma_{k+1}^2}2\right)\notag\\
&\quad\leq (T+1)\left(M_1M_2(N_1+1)C_T+\frac12\right)(\gamma_{m(t)} \vee \epsilon_{m(t)}).\label{eq:APTProof2}
\end{align}
Thus, combining \eqref{eq:APTProof4}, \eqref{eq:APTProof1} and \eqref{eq:APTProof2} yields
\begin{equation}
\sup_{s\leq T}d_{\mathscr F}(\mu_{t+s},\Phi(\mu_t,s))\leq C'_T(\gamma_{m(t)} \vee \epsilon_{m(t)}),
\label{eq:APTProof3}\end{equation}
with $C'_T=\frac32+(T+1)\left(M_1M_2(N_1+1)C_T+\frac12\right)$. Then, $(\mu_t)_{t\geq0}$ is an asymptotic pseudotrajectory of $\Phi$ (with respect to $d_{\mathscr F}$).

Now, we turn to the study of the case $\lambda(\gamma,\epsilon)>0$. For any $\lambda<\lambda(\gamma,\epsilon)$, we have (for $n$ large enough) $\gamma_n\vee\epsilon_n \leq\exp(-\lambda\tau_n)$. Then, for any $t$ large enough,
\[\gamma_{m(t)}\vee\epsilon_{m(t)}\leq\e^{-\lambda\tau_{m(t)}}
\leq\e^{\lambda(t-\tau_{m(t)})}\e^{-\lambda t}
\leq\e^{\lambda(\gamma,\epsilon)}\e^{-\lambda t}.\]
Now, plugging this upper bound in \eqref{eq:APTProof3}, we get, for $\lambda<\lambda(\gamma,\epsilon)$,
\begin{equation}
\sup_{s\leq T}d_{\mathscr F}(\mu_{t+s},\Phi(\mu_t,s))\leq \e^{\lambda(\gamma,\epsilon)}C'_T\e^{-\lambda t}.
\label{eq:APTProof5}
\end{equation}
Finally, we can deduce that
\[\limsup_{t\to+\infty}\frac1t\log\left(\sup_{0\leq s\leq T}d(\mu_{t+s},\Phi(\mu_t, s))\right)\leq-\lambda\]
for any $\lambda<\lambda(\gamma,\epsilon)$, which concludes the proof of Theorem~\ref{theorem:APTrPT}.
\end{proof}

\begin{proof}[Proof of Theorem~\ref{theorem:SpeedConv}]The first part of the proof is an adaptation of \cite{Ben99}. Assume Assumption~\ref{assumption:ergodicity}.i) and, without loss of generality, assume $M_3>1$. If $v>\lambda(\gamma,\epsilon)$, fix $\varepsilon>v-\lambda(\gamma,\epsilon)$, otherwise let $\varepsilon>0$, and set $u:=v-\varepsilon$, $T_\varepsilon:=\varepsilon^{-1}\log M_3$. Since $u<\lambda(\gamma,\epsilon)$, and using \eqref{eq:APTProof5}, the following sequence of inequalities holds, for any $T\in[T_\varepsilon,2T_\varepsilon]$ and $n\in\N$:
\begin{align*}
d_{\mathscr{G}}\left(\mu_{(n+1)T},\pi\right)
&\leq d_{\mathscr{G}}\left(\mu_{(n+1)T},\Phi(\mu_{nT},T)\right)+ d_{\mathscr{G}}\left(\Phi(\mu_{nT},T),\pi\right)\\
	&\leq \e^{\lambda(\gamma,\epsilon)}C'_{T}\e^{-unT}+M_3d_{\mathscr{G}}\left(\mu_{nT},\pi\right)\e^{-vT}\\
	&\leq \e^{\lambda(\gamma,\epsilon)}C'_{T}\e^{-unT}+d_{\mathscr{G}}\left(\mu_{nT},\pi\right)\e^{-uT},
\end{align*}
with $C'_T=\frac32+(T+1)\left(M_1M_2(N_1+1)C_T+\frac12\right)$. Denoting by $\delta_n:=d_{\mathscr{G}}\left(\mu_{nT},\pi\right)$ and $\rho:=\e^{-uT}$, the previous inequality turns into $\delta_{n+1}\leq \e^{\lambda(\gamma,\epsilon)}C'_{T}\rho^n+\rho\delta_n$, from which we derive
\[\delta_n\leq n\rho^{n-1}C'_{T}\e^{\lambda(\gamma,\epsilon)}+\rho^n\delta_0.\]
Hence, for every $n\geq0$ and $T\in[T_\varepsilon,2T_\varepsilon]$, we have
\[d_{\mathscr{G}}\left(\mu_{nT},\pi\right)\leq\e^{-(u-\varepsilon)nT}\left(M_5+d_{\mathscr{G}}\left(\mu_0,\pi\right)\right),\quad M_5=\e^{\lambda(\gamma,\epsilon)}\left(\sup_{n\geq0}n\e^{-\varepsilon nT}\right)\left(\sup_{T\in[T_\varepsilon,2T_\varepsilon]}C'_T\right).\]
Then, for any $t>T_\varepsilon$, let $n=\lfloor tT_\varepsilon^{-1}\rfloor$ and $T=tn^{-1}$. Then, $T\in[T_\varepsilon,2T_\varepsilon]$ and the following upper bound holds:
\[d_{\mathscr{G}}\left(\mu_t,\pi\right)\leq\left(M_5+d_{\mathscr{G}}\left(\mu_0,\pi\right)\right)\e^{-(u-\varepsilon)t}.\]

Now, assume Assumption~\ref{assumption:ergodicity}.ii). For any (small) $\varepsilon>0$, there exists $\e^{\lambda(\gamma,\epsilon)}$ such that $\gamma_{m(t)}\vee\epsilon_{m(t)}\leq \e^{\lambda(\gamma,\epsilon)}\exp(-(\lambda(\gamma,\epsilon)-\varepsilon)t)$. For any $\alpha\in(0,1)$, we have
\begin{align}
d_{\mathscr F\cap\mathscr G}(\mu_t,\pi)
&\leq d_{\mathscr F\cap\mathscr G}(\mu_t,\Phi(\mu_{\alpha t},(1-\alpha)t))+ d_{\mathscr F\cap\mathscr G}(\Phi(\mu_{\alpha t},(1-\alpha)t),\pi)\notag\\
&\leq C'_{(1-\alpha)t}(\gamma_{m(\alpha t)}\vee\epsilon_{m(\alpha t)})+M_3\e^{-v(1-\alpha)t}\notag\\
&\leq M_4\e^{r(1-\alpha)t} \e^{\lambda(\gamma,\epsilon)}\e^{-(\lambda(\gamma,\epsilon)-\varepsilon)\alpha t}+M_3\e^{-v(1-\alpha)t}.\label{eq:SpeedConvProof1}
\end{align}
Optimizing \eqref{eq:SpeedConvProof1} by taking $\alpha=(r+v)(r+v+\lambda(\gamma,\epsilon)-\varepsilon)^{-1}$, we get
\[d_{\mathscr F\cap\mathscr G}(\mu_t,\pi)\leq M_5\exp\left(-\frac{v(\lambda(\gamma,\epsilon)-\varepsilon)}{r+v+\lambda(\gamma,\epsilon)-\varepsilon}t\right),\]
with $M_5=M_4\e^{\lambda(\gamma,\epsilon)}+M_3$, which depends on $\varepsilon$ only through $M_3$.

Lastly, assume Assumption~\ref{assumption:convergence}.iii). Denote by $\mathcal K$ the set of probability measures $\nu$ such that
\[\nu(W)<M=\sup_{n\geq0} \E[W(y_n)].\]
Let $\varepsilon>0$ and $K=\{x\in\R^D:W(x)\leq M/\varepsilon\}$. For every $\nu\in \mathcal K$, using Markov's inequality, it is clear that
\[\nu(K^C)\leq\frac{\varepsilon}{M}\nu(W)\leq\varepsilon.\]
Then $\mathcal K$ is a relatively compact set (by Prokhorov's Theorem). The measure $\pi$ is an attractor in the sense of \cite{Ben99}, which means that $\lim_{t\to+\infty}d_{\mathscr G}(\Phi(\nu,t),\pi)=0$ uniformly in $\nu\in\mathcal K$. Then, since for any $t>0,\mu_t\in \mathcal K$, we can apply \cite[Theorem~6.10]{Ben99} to achieve the proof.
\end{proof}

\begin{proof}[Proof of Theorem~\ref{theorem:functionalCV}]
We shall prove the convergence of the sequence of processes $(Y^{(t)}_s)_{0\leq s\leq T}$ , as $t\to+\infty$, toward $(X^\pi_s)_{0\leq s\leq T}$ in the Skorokhod space $D([0,T])$, for any $T>0$. Then, using \cite[Theorem~16.7]{Bil99}, this convergence entails Theorem~\ref{theorem:functionalCV}, i.e. convergence of the sequence $(Y^{(t)})$ in $D([0,\infty))$.

Let $T>0$. The proof of functional convergence classically relies on proving the convergence of finite-dimensional distributions, on the one hand, and tightness, on the other hand. First, we prove the former, which is the first part of Theorem~\ref{theorem:functionalCV}. We choose to prove the convergence of the finite-dimensional distributions in the case $m=2$. The proof for the general case is similar but with a laborious notation. Denote by $T_{u,v}g(y):=\E[g(Y_v)|Y_u=y]$. With this notation, \eqref{eq:APTProof3} becomes
\[\sup_{s\leq T}\sup_{g\in\mathscr F}\left(\mu_tT_{t,t+s}g-\mu_tP_sg\right)\leq C'_T(\gamma_{m(t)} \vee \epsilon_{m(t)}).\]
This upper bound does not depend on $\mu_t$, so, for any probability distribution $\nu$, we have
\[
\sup_{s\leq T}\sup_{g\in\mathscr F}\left(\nu T_{t,t+s}g-\nu P_sg\right)\leq C'_T(\gamma_{m(t)} \vee \epsilon_{m(t)}).
\]
This inequality implies that, for any $\nu$,
\begin{equation}\sup_{s_1\leq s_2\leq T}\sup_{g\in\mathscr F}\left(\nu T_{t+s_1,t+s_2}g-\nu P_{s_2-s_1}g\right)\leq C'_T(\gamma_{m(t)} \vee \epsilon_{m(t)}),
\label{eq:FunctionalCVProof1}
\end{equation}
which converges toward 0 as $t\to+\infty$. From now on, we denote, for any function $f$, $\widehat f_x(y):=f(x,y)$. If $f$ is a smooth function (say in $\mathscr C^\infty_c$ with enough derivatives bounded), $\hat f_\cdot(\cdot)\in\mathscr F$. On the one hand, for $0,s_1<s_2<T$,
\[\E[f(X^\pi_{s_1},X^\pi_{s_2})]=\int P_{s_2-s_1}\widehat f_y(y)\pi(dy)=\pi P_{s_2-s_1}\widehat f_\cdot(\cdot).\]
On the other hand, we have
\begin{align*}
\E[f(Y^{(t)}_{s_1},Y^{(t)}_{s_2}]&=\E\left[\E[f(Y^{(t)}_{s_1},Y^{(t)}_{s_2}|Y^{(t)}_{s_1}]\right]
	=\E\left[T_{t+s_1,t+s_2}\widehat f_{Y_{t+s_1}}(Y_{t+s_1})\right]\\
	&=T_{0,t+s_1}\left(T_{t+s_1,t+s_2}\widehat f_\cdot(\cdot)\right).
\end{align*}
We have the following triangle inequality:
\begin{align}
\left|\E[f(Y^{(t)}_{s_1},Y^{(t)}_{s_2}]- \E[f(X^\pi_{s_1},X^\pi_{s_2})]\right|
	&=\left|T_{0,t+s_1}\left(T_{t+s_1,t+s_2}\widehat f_\cdot(\cdot)\right)-\pi P_{s_2-s_1}\widehat f_\cdot(\cdot)\right|\notag\\
	&\leq\left|T_{0,t+s_1}\left(T_{t+s_1,t+s_2}\widehat f_\cdot(\cdot)- P_{s_2-s_1}\widehat f_\cdot(\cdot)\right)\right|\notag\\
	&\quad+\left|T_{0,t+s_1}\left(P_{s_2-s_1}\widehat f_\cdot(\cdot)\right)-\pi P_{s_2-s_1}\widehat f_\cdot(\cdot)\right|\label{eq:FunctionalCVProof2}
\end{align}
Firstly, using \eqref{eq:FunctionalCVProof1}, and if $\widehat f_\cdot(\cdot)\in\mathscr F$,
\[\lim_{t\to\infty}T_{0,t+s_1}\left(T_{t+s_1,t+s_2}\widehat f_\cdot(\cdot)- P_{s_2-s_1}\widehat f_\cdot(\cdot)\right)=\lim_{t\to\infty}\mu_{t+s_1}\left(T_{t+s_1,t+s_2}\widehat f_\cdot(\cdot)- P_{s_2-s_1}\widehat f_\cdot(\cdot)\right)=0.\]
Secondly, $P_{s_2-s_1}f_\cdot(\cdot)\in\mathscr C^0_b$ and, using Theorem~\ref{theorem:SpeedConv},
\[\lim_{t\to\infty}T_{0,t+s_1}\left(P_{s_2-s_1}\widehat f_\cdot(\cdot)\right)-\pi P_{s_2-s_1}\widehat f_\cdot(\cdot)=0.\]
From \eqref{eq:FunctionalCVProof2}, it is straightforward that, for a smooth $f$,
\[\lim_{t\to\infty}\left|\E[f(Y^{(t)}_{s_1},Y^{(t)}_{s_2}]- \E[f(X^\pi_{s_1},X^\pi_{s_2})]\right|=0,\]
and applying Lemma~\ref{lemma:WeakConv} achieves the proof of finite dimensional convergence for $m=2$.

To prove tightness, which is the second part of Theorem~\ref{theorem:functionalCV}, we need the following lemma, whose proof is postponed to Section~\ref{section:appendix}.

\begin{lemma}[Martingale properties]
Let $f$ be a continuous and bounded function. The process $(\widehat{M}_n^f)_{n\geq 0}$, defined for every $n\geq 0$ by
\[\widehat{M}^{f}_n = f(y_n) - f(y_0) - \sum_{k=0}^{n-1} \gamma_{k+1} \mathcal{L}_k f(y_k),\]
is a martingale, with
\[\langle \widehat{M}^{f} \rangle_n= \sum_{k=0}^{n-1} \gamma_{k+1}\Gamma_kf(y_k).\]

Moreover, under Assumption~\ref{assumption:variance}, if $d\geq d_2$ then for every $N\geq 0$, there exist a constant $M_7>0$ (depending on $N$ and $y_0$) such that
\[\mathbb{E}\left[\sup_{n \leq N} \chi_{d_1} (y_n)\right] \leq M_7.\]
\label{lemma:mart}
\end{lemma}

Now, define
\begin{align*}
M^{(t,i)}_s &= \widehat{M}_{m(t+s)}^{\varphi_i} -\widehat{M}_{m(t)}^{\varphi_i},\\
A^{(t,i)}_s
&= \varphi_i(Y_t) + \int_{\tau_{m(t)}}^{\tau_{m(t+s)}} \mathcal{L}_{m(u)} \varphi_i(Y_u) du=\varphi_i(y_{m(t)}) + \sum_{k=m(t)}^{m(t+s)-1} \gamma_{k+1} \mathcal{L}_k \varphi_i (y_k)
\end{align*}
and
$$
Y^{(t,i)}_s= \varphi_i (Y^{(t)}_s).
$$
With this notation and Lemma~\ref{lemma:mart}, we have
$$
Y^{(t,i)}_s = A^{(t,i)}_s + M^{(t,i)}_s
$$
and $(M^{(t,i)}_s)_{s\geq 0}$ is a martingale with quadratic variation
$$
\langle M^{(t,i)} \rangle_s = \int_{\tau_{m(t)}}^{\tau_{m(t+s)}} \Gamma_{m(u)} \varphi_i(Y_u) du,
$$
where $\Gamma_n$ is as in Assumption~\ref{assumption:variance}. From the convergence of finite-dimensional distributions, for every $s\in [0,T]$, the sequence $(Y^{(t)}_s)_{t\geq 0}$ is tight. It is then enough, from the Aldous-Rebolledo criterion (see Theorems~2.2.2 and 2.3.2 in \cite{JM86}) and Lemma \ref{lemma:mart} to show that: for every $S \geq 0$, $ \varepsilon, \eta > 0$, there exists a $\delta>0$ and $t_0>0$ with the property that whatever the family of stopping times $(\sigma^{(t)})_{t\geq0}$, with $\sigma^{(t)} \leq S$, for every $i \in \{1,\dots D\}$,
\begin{equation}
\label{eq:Aldous-varquad}
\sup_{t\geq t_0} \sup_{\theta \leq \delta} \ \mathbb{P} \left(  \left| \langle M^{(t,i)} \rangle_{\sigma^{(t)}} - \langle M^{(t,i)} \rangle_{\sigma^{(t)} + \theta}  \right|  \geq \eta \right) \leq \varepsilon
\end{equation}
and
\begin{equation}
\label{eq:Aldous-varfinie}
\sup_{t\geq t_0} \sup_{\theta \leq \delta} \ \mathbb{P} \left( \left| A^{(t,i)}_{\sigma^{(t)}} -A^{(t,i)}_{\sigma^{(t)} + \theta}  \right|  \geq \eta \right) \leq \varepsilon. 
\end{equation}
We have, using Assumption~\ref{assumption:variance},
\begin{align*}
A^{(t,i)}_{\sigma^{(t)} + \theta} -A^{(t,i)}_{\sigma^{(t)}} &=\int_{\tau_{m(t+\sigma^{(t)})}}^{\tau_{m(t+\sigma^{(t)} + \theta)}} \mathcal L_{m(u)}\varphi_i(Y_u) du\leq \int_{\tau_{m(t+\sigma^{(t)})}}^{\tau_{m(t+\sigma^{(t)} + \theta)}} M_6 \chi_{d_2} (Y_u) du\\
&\leq M_6  | \tau_{m(t+\sigma^{(t)} + \theta)}- \tau_{m(t+\sigma^{(t)})}| \sup_{r \leq T}  \chi_{d_2} (Y_r).
\end{align*}
From the definition of $\tau_n$,
$$
|\tau_{m(t+\sigma^{(t)} + \theta)}- \tau_{m(t+\sigma^{(t)})}| \leq \theta + \gamma_{m(t)+1},
$$
and then, using Lemma~\ref{lemma:mart} and Markov's inequality
$$
\mathbb{P} \left( \left| A^{(t,i)}_{\sigma^{(t)}} -A^{(t,i)}_{\sigma^{(t)} + \theta}  \right|  \geq \eta \right) \leq \frac{ M_6(\theta + \gamma_{m(t_0)+1})}{\eta} \mathbb{E}[\sup_{s \leq T}  \chi_{d_2} (Y_r)] \leq M_6M_7  \frac{ (\delta + \gamma_{m(t_0)+1})}{\eta}. 
$$
Proving the inequality \eqref{eq:Aldous-varquad} is done in a similar way, and achieves the proof.
\end{proof}

\section{Appendix}
\label{section:appendix}

\subsection{General appendix}
\begin{lemma}[Weak convergence and $d_\mathscr{F}$]
Assume that $\mathscr F$ is a star domain with respect to 0 (i.e. if $f\in\mathscr F$ then $\lambda f\in\mathscr F$ for $\lambda\in[0,1]$). Let $(\mu_n),\mu$ be probability measures. If $\lim_{n\to\infty}d_\mathscr{F}(\mu_n,\mu)=0$ and, for every $g\in\mathscr C^\infty_c$, there exists $\lambda>0$ such that $\lambda g\in\mathscr F$, then $(\mu_n)$ converges weakly toward $\mu$. If $\mathscr F\subseteq\mathscr C^1_b$, then $d_{\mathscr F}$ metrizes the weak convergence.
\label{lemma:WeakConv}\end{lemma}

\begin{proof}
Let $f\in\mathscr C^0_b,g\in\mathscr C^\infty_c$. Note that $fg\in\mathscr C^0_c$ and, using Weierstrass' Theorem, it is well known that, for all $\varepsilon>0$, there exists $\varphi\in \mathscr C^\infty_c$ such that $\|fg-\varphi\|_\infty\leq\varepsilon$. By hypothesis, and since $\mathscr F$ is a star domain, there exists $\lambda>0$ such that $\lambda g,\lambda\varphi\in\mathscr F$. Then,
\[\left|\mu_n(fg)-\mu(fg)\right|\leq
\left|\mu_n(fg)-\mu_n(\varphi)\right|+
\frac1\lambda\left|\mu_n(\lambda\varphi)-\mu(\lambda\varphi)\right|+
\left|\mu(fg)-\mu(\varphi)\right|,\]
thus $\limsup_{n\to\infty}\left|\mu_n(fg)-\mu(fg)\right|\leq2\varepsilon$. Now,
\begin{align*}
\left|\mu_n(f)-\mu(f)\right|&\leq
\left|\mu_n(f-fg)-\mu(f-fg)\right|+\left|\mu_n(fg)-\mu(fg)\right|\\
&\leq\|f\|_\infty\left|\mu_n(1-g)-\mu(1-g)\right|+\left|\mu_n(fg)-\mu(fg)\right|\\
&\leq\frac{\|f\|_\infty}\lambda\left|\mu_n(\lambda g)-\mu(\lambda g)\right|+\left|\mu_n(fg)-\mu(fg)\right|
\end{align*}
so that $\limsup_{n\to\infty}\left|\mu_n(f)-\mu(f)\right|\leq2\varepsilon$, for any $\varepsilon>0$, which concludes the proof.

Now, assuming $\mathscr F\subseteq\mathscr C_b^1$, use \cite[Theorem~5.6]{Che04}. Then, convergence with respect to $d_{\mathscr F}$ is equivalent to weak convergence. Indeed, $d_{\mathscr C_b^1}$ is the well-known Fortet-Mourier distance, which metrizes the weak topology. It is also the Wasserstein distance $\Wass_\delta$, with respect to the distance $\delta$ such that
\[\forall x,y\in\R^D,\quad\delta(x,y)=\sup_{f\in\mathscr C_b^1}|f(x)-f(y)|=|x-y|\wedge 2.\]
See also \cite[Theorem~4.4.2.]{RKSF13}.
\end{proof}


\begin{proof}[Proof of Lemma~\ref{lemma:mart}]
Let $\mathscr{F}_n=\sigma(y_0,\dots,y_n)$ be the natural filtration. Classically, we have
\begin{align*}
\mathbb{E}[ \widehat{M}^{f}_{n+1} \ | \ \mathscr{F}_n ] 
&= \mathbb{E}[ f(y_{n+1}) - f(y_0) - \sum_{k=0}^{n} \gamma_{k+1} \mathcal{L}_{k} f(y_k) \ | \ \mathscr{F}_n ] \\
&= f(y_n) + \gamma_{n+1} \mathcal{L}_n f(y_n) - f(y_0) - \sum_{k=0}^{n} \gamma_{k+1} \mathcal{L}_{k} f(y_k)\\
&= \widehat{M}^{f}_n.
\end{align*}
Moreover,
\begin{align*}
\mathbb{E}[ (\widehat{M}^{f}_{n+1})^2 \ | \ \mathscr{F}_n ] 
&= \mathbb{E}\left[ \left.f(y_{n+1})^2 + f(y_0)^2 + \left(\sum_{k=0}^{n} \gamma_{k+1} \mathcal{L}_{k} f(y_k)\right)^2 \ \right| \ \mathscr{F}_n \right] \\
&\quad  \ - \mathbb{E}\left[\left.  2 f(y_{n+1}) \left( f(y_0) + \sum_{k=0}^{n} \gamma_{k+1} \mathcal{L}_{k} f(y_k)\right) \ \right| \ \mathscr{F}_n \right] \\
&\quad  \ + \mathbb{E}\left[\left. 2 f(y_0) \left(\sum_{k=0}^{n} \gamma_{k+1} \mathcal{L}_{k} f(y_k)\right) \ \right| \ \mathscr{F}_n \right] \\
&= f(y_{n})^2 + \gamma_{n+1} \mathcal{L}_n f^2(y_n) + f(y_0)^2 + \left(\sum_{k=0}^{n} \gamma_{k+1} \mathcal{L}_{k} f(y_k) \right)^
2 \\
&\quad  \ -  2 (f(y_{n}) + \gamma_{n+1} \mathcal{L}_n f(y_n)) \left( f(y_0) + \sum_{k=0}^{n} \gamma_{k+1} \mathcal{L}_{k} f(y_k)\right)  \\
&\quad  \ + 2 f(y_0) \left(\sum_{k=0}^{n} \gamma_{k+1} \mathcal{L}_{k} f(y_k)\right).
\end{align*}
Henceforth,
\begin{align*}
\mathbb{E}[ (\widehat{M}^{f}_{n+1})^2 \ | \ \mathscr{F}_n ]&=\gamma_{n+1} \mathcal{L}_n f^2(y_n) + 2 \gamma_{n+1} \mathcal{L}_{n} f(y_n) \left(\sum_{k=0}^{n-1} \gamma_{k+1} \mathcal{L}_{k} f(y_k) \right) +  (\gamma_{n+1} \mathcal{L}_{n} f(y_n))^2\\
&\quad  \ -  2 f(y_{n})\gamma_{n+1} \mathcal{L}_{n} f(y_n) -2 \gamma_{n+1} \mathcal{L}_n f(y_n) \left( f(y_0) + \sum_{k=0}^{n} \gamma_{k+1} \mathcal{L}_{k} f(y_k)\right)  \\
&\quad  \ + 2 f(y_0)  \gamma_{n+1} \mathcal{L}_{n} f(y_n) + (m_n^f)^2 \\
&=(\widehat{M}_n^f)^2 + \gamma_{n+1} \mathcal{L}_n f^2(y_n) - (\gamma_{n+1} \mathcal{L}_{n} f(y_n))^2 -  2 f(y_{n})\gamma_{n+1} \mathcal{L}_{n} f(y_n)\\
&=(\widehat{M}_n^f)^2+\gamma_{n+1}\Gamma_nf.
\end{align*}

Now, on the first hand, using Assumption~\ref{assumption:variance},
\begin{align*}
\mathbb{E}\left[ \langle \widehat{M}^{\chi_{d_2}} \rangle_N \right]
&= \mathbb{E}\left[ \sum_{k=0}^{N-1} \gamma_{k+1} \Gamma_{k+1} \chi_{d_2}(y_k) \right]
\leq M_6 \sum_{k=0}^{N-1} \gamma_{k+1}   \mathbb{E}\left[  \chi_{d}(y_k) \right]
\leq  M_2M_6 \sum_{k=0}^{N-1}   \gamma_{k+1},
\end{align*}
and then Doob's inequality gives
$$
\mathbb{E}\left[\left(\sup_{n \leq N} \widehat{M}_n^{\chi_{d_2}}\right)^{2}\right]^{1/2}  \leq 2 \mathbb{E}\left[ \langle \widehat{M}^{\chi_{d_2}} \rangle_N \right]^{1/2} \leq C,
$$
for some constant $C$, only depending on $N$. On the other hand, from Lemma~\ref{lemma:mart} and Assumption~\ref{assumption:variance},
$$
\sup_{n \leq N} \chi_{d_2}(y_n) \leq \chi_{d_2}(y_0) +M_6\sum_{k=0}^{N-1}  \gamma_{k+1}\sup_{n\leq k} \chi_{d_2}(y_n) + \sup_{n \leq N} \widehat{M}_n^{\chi_{d_2}}.
$$
Using the triangle inequality, we then have
\begin{align*}
\mathbb{E}\left[\left(\sup_{n \leq N} \chi_{d_2}(y_n) \right)^{2}\right]^{1/2}
&\leq  
\mathbb{E}\left[\left(\chi_{d_2}(y_0) \right)^{2}\right]^{1/2}
+ M_6 \sum_{k=0}^{N-1} \gamma_{k+1}
\mathbb{E}\left[\left(\sup_{n\leq k} \chi_{d_2}(y_n) \right)^{2}\right]^{1/2}\\
&\quad+ \mathbb{E}\left[\left(\sup_{n \leq N} \widehat{M}_n^{\chi_{d_2}}\right)^{2}\right]^{1/2}.
\end{align*}
Then, using (discrete) Gr\"onwall's Lemma as well as Cauchy-Schwarz's inequality ends the proof.
\end{proof}

\subsection{Appendix for the penalized bandit algorithm}

\begin{proof}[Proof of Proposition~\ref{proposition:regPDMP}]
The unique solution of the ordinary differential equation $y'(t)=a-by(t)$ with initial condition $x$ is given by
\[\Psi(x,t) =
\left\{ \begin{array}{ll}
\left(x-\frac ab\right)\e^{-bt}+\frac ab &\text{if }b>0\\
x+at&\text{if }b=0
\end{array} \right..\]

Firstly, assume that $b>0$ and let $t\in[0,T]$. We have, for $x>0$
\begin{align}
&P_t f(x) 
= \mathbb{E}_x\left[ f(X_t) \right]= f\left(\Psi(x,t)\right) \mathbb{P}_x\left( T> t \right) + \mathbb{E}_x\left[ f(X_t)|T\leq t \right]\mathbb{P}_x\left( T\leq t \right)\notag\\
&\quad= f\left(\Psi(x,t)\right) \exp\left(-\int_0^t (c+d\Psi(x,s)) ds \right)\notag\\
&\quad+ \int_0^t P_{t-u} f(\Psi(x,u)+1) (c+d\Psi(x,u))\exp\left(-\int_0^u (c+d\Psi(x,s))  ds \right) du.\label{eq:proofPropregPDMP}
\end{align}
At this stage, the smoothness of the right-hand side of \eqref{eq:proofPropregPDMP} with respect to $x$ is not clear. Let $0<\varepsilon<\min(a/b,1/2)$. If $0\leq x\leq a/b-\varepsilon$, use the substitution
\[v=\Psi(x,u),\quad u=\varphi(x,v)=\frac1b\log\left(\frac{x-\frac ab}{v-\frac ab}\right),\]
to get
\begin{align*}
P_t f(x) 
&= f\left(\Psi(x,t)\right) \exp\left(-\int_0^t (c+d\Psi(x,s)) ds \right)\\
&\quad+ \int_x^{\Psi(x,t)} P_{t-\varphi(x,v)} f(v +1) \exp\left(-\int_0^{\varphi(x,v)} (c+d\Psi(x,s))  ds \right)\frac{c+dv}{a-bv}dv.
\end{align*}
Note that $\Psi(x,t)\leq \Psi(a/b-\varepsilon,t)<a/b$, so that $a-bv\neq0$. Since $s\mapsto P_sf(x)$, $\Psi$, $\varphi$ and $f$ are smooth, $x\mapsto P_tf(x)\in\mathscr C^N([o,a/b-\varepsilon])$. The reasoning holds with the same substitution for $x\geq a/b+\varepsilon$, so that $P_tf\in\mathscr C^N(\R_+\backslash\{a/b\})$. Now, if $x>a/b-\varepsilon$, for any $u>0$,
\[\Psi(x,u)+1\geq a/b+1-\varepsilon\geq a/b+\varepsilon,\]
so $x\mapsto P_{t-u}f(\Psi(x,u)+1)$ is smooth. Thus the right-hand side of \eqref{eq:proofPropregPDMP} is smooth as well and $P_tf\in\mathscr C^N(\R_+)$.

Now, let us show that the semigroup generated by $\mathcal{L}$ has bounded derivatives. Note that it is possible to mimic this proof for the example of the WRW treated in Section~\ref{subsection:randomwalk} when the derivatives of $P_tf$ are not explicit. Let $\mathcal A_nf=f^{(n)}$, $\mathcal Jf(x)=f(x+1)-f(x)$ and $\psi_n(s) = P_{t-s}\mathcal A_n P_sf$ for $0\leq n\leq N$. So, $\psi_n'(s) = P_{t-s} (\mathcal A_n \mathcal{L} - \mathcal{L}\mathcal A_n ) P_s f$. It is clear that $\mathcal A_{n+1}=\mathcal A_1\mathcal A_n$, that $\mathcal A_n\mathcal J=\mathcal J\mathcal A_n$ and that
\[\mathcal{L} g(x)= (a-bx) \mathcal A_1g(x) +(c+dx) \mathcal Jg(x).\]
It is straightforward by induction that
\[\mathcal A_n\mathcal L g= \mathcal L \mathcal A_n g - nb \mathcal A_ng + nd\mathcal J\mathcal A_{n-1}g,\]
so the following inequality holds:
\[\left(\mathcal A_n\mathcal L-\mathcal L\mathcal A_n\right)g\leq-nb\mathcal A_ng+2|d|n\|\mathcal A_{n-1}g\|_\infty.\]
Hence,
\[\psi_n'(s) \leq -nb \psi_n(s) + 2|d|n\|\mathcal A_{n-1}P_sf\|_\infty.\]
In particular, $\psi_1'(s) \leq -b \psi_1(s) +2d\|f\|_\infty$, so, by Gr\"onwall's inequality,
\[\psi_1(s) \leq \left(\psi_1(0)-\frac{2|d|}b\|f\|_\infty\right)e^{-bs}+\frac{2|d|}b\|f\|_\infty\leq \|f'\|_\infty+\frac{2d}b\|f\|_\infty.\]
Let us show by induction that
\begin{equation}
\psi_n(s)\leq\sum_{k=0}^n{\left(\frac{2|d|}b\right)^{n-k}\|f^{(k)}\|_\infty}.
\label{eq:InducProofProp}
\end{equation}
If \eqref{eq:InducProofProp} is true for some $n\geq1$ (we denote by $K_n$ its right-hand side), then for all $t<T$, $\psi_n(t)\leq K_n$ and, since $\mathcal A_nP_t(-f)=-\mathcal A_nP_tf$, $|\psi_n(t)|\leq K_n$, so $\|\mathcal A_nP_sf\|_\infty\leq K_n$. Then, we deduce that $\psi_{n+1}'(s)\leq-(n+1)b\psi_{n+1}(s)+2(n+1)dK_n$. Use Gr\"onwall's inequality once more to have $\psi_{n+1}(s)\leq K_{n+1}$ and achieve the proof by induction. In particular, taking $s=t$ in \eqref{eq:InducProofProp} provides $\mathcal A_nP_tf\leq K_n$
and, since $\mathcal A_nP_t(-f)=-\mathcal A_nP_tf$, $\mathcal A_nP_tf\leq K_n$. As a conclusion, for $n\in\{0,\dots,N\},$
\[\|\mathcal (P_tf)^{(n)}\|_\infty\leq\sum_{k=0}^n{\left(\frac{2|d|}b\right)^{n-k}\|f^{(k)}\|_\infty},\]
which concludes the proof when $b>0$.

The case $b=0$ is dealt with in a similar way. We use the substitution $\varphi(x,v)=(v-x)/a$ in \eqref{eq:proofPropregPDMP}, which is enough to prove smoothness (this time, $\Psi(x,\cdot)$ is a diffeomorphism for any $x\geq0$), and it is easy to mimic the proof to obtain the following estimates, for $s\leq t$,
\[|\psi_{n}(s)|\leq\sum_{k=0}^n\frac{n!}{k!}(2|d|T)^{n-k}\|f^{(k)}\|_\infty.\]
\end{proof}

\begin{proof}[Proof of Lemma~\ref{lemma:truncationAPT}]
First, we shall prove that Assumption~\ref{assumption:convergence} holds; let
\[y\in\text{Supp}(\mathscr L(y_n^{(l,\delta)}))=[0,\delta\sqrt n].\]
Note that $\widetilde I_n^0(y),I_n^0(y)\leq1$ and $\widetilde I_n^1(y),I_n^1(y)\leq0$, so if $y_n^{(l,\delta)}\leq\delta\gamma_{n+1}^{-1}-1$, then $y_{n+1}^{(l,\delta)}\leq\delta\gamma_{n+1}^{-1}$. For $f\in\mathscr F$,
\begin{align*}
&|\mathcal L_n^{(l,\delta)}f(y)-\mathcal L_nf(y)|\leq\gamma_{n+1}^{-1}\E\left[\left.f(y_{n+1}^{(l,\delta)})-f(y_{n+1})\right|y_n=y_n^{(l,\delta)}=y\right]\\
&\quad\leq\frac{\indic_{y\geq\delta\gamma_{n+1}^{-1}-1}}{\gamma_{n+1}}\Big(p_0(1-\gamma_ny)\left|f(\delta\gamma_{n+1}^{-1})-f(y+I_n^0(y))\right|\\
&\quad\quad+\widetilde p_0(1-\gamma_ny)\left|f(\delta\gamma_{n+1}^{-1})-f(y+\widetilde I_n^0(y))\right|\Big)\\
&\quad\leq\frac{\|f'\|_\infty\indic_{y\geq\delta\gamma_{n+1}^{-1}-1}}{\gamma_{n+1}}\left(p_0(1-\gamma_ny)+\widetilde p_0(1-\gamma_ny)\right)\leq\frac{y+1}\delta\|f'\|_\infty\indic_{y\geq\delta\gamma_{n+1}^{-1}-1}\\
&\quad\leq\frac{(y+1)^2}{\delta^2}\|f'\|_\infty\gamma_{n+1}.
\end{align*}
Using this inequality with \eqref{eq:PBP5}, we can explicit the convergence of $\mathcal L_n^{(l,\delta)}$ toward $\mathcal L$ defined in \eqref{eq:PBP6}:
\begin{align}
|\mathcal L_n^{(l,\delta)}f(y)-\mathcal Lf(y)|&\leq |\mathcal L_n^{(l,\delta)}f(y)-\mathcal L_nf(y)|+|\mathcal L_nf(y)-\mathcal Lf(y)|\notag\\
&=\chi_3(y)(\|f\|_\infty+\|f'\|_\infty+\|f''\|_\infty)\mathcal O(\gamma_n).
\label{eq:proofTruncatedChain1}
\end{align}
Note that the notation $\mathcal O$ depends here on $l$ and $\delta$, but is uniform over $y$ and $f$.

Assumption~\ref{assumption:regularity} holds, since it takes into account only the limit process generated by $\mathcal L$, and it is a consequence of Proposition~\ref{proposition:regPDMP}: for $n\leq3$,
\[\|\mathcal (P_tf)^{(n)}\|_\infty\leq\sum_{k=0}^n{\left(\frac{2|p_0'(1)|}{p_1(1)}\right)^{n-k}\|f^{(k)}\|_\infty}.\]

Now, we shall check a Lyapunov criterion for the chain $(y_n^{(l,\delta)})_{n\geq0}$, in order to ensure Assumption~\ref{assumption:moments}. Taking $V(y)= \e^{\theta y}$, where (small) $\theta>0$ will be chosen afterwards, we have, for $n\geq l$ and $y\leq\delta\gamma_n^{-1}$,
\begin{align*}
\mathcal L_n^{(l,\delta)} V(y) &\leq\gamma_{n+1}^{-1}\E\left[V((y+I_n(y))\wedge\delta\sqrt n)-V(y)\right]\leq\gamma_{n+1}^{-1}\E\left[V(y+I_n(y))-V(y)\right]\\
&\leq V(y)  \sqrt{n+1} \left( \E[\e^{\theta I_n(y)}] - 1\right).
\end{align*}
Let $\varepsilon>0$; we are going to decompose $I_n(y)$. The first term is
\begin{align*}
&\sqrt{n+1} \left(\exp\left(\frac{\sqrt{n+1} - \sqrt{n} - 1}{\sqrt{n}} \theta y \right) - 1\right) p_1(1 - \gamma_n y)\\
&\quad\leq\sqrt{n+1} \left(\frac{\sqrt{n+1} - \sqrt{n} - 1}{\sqrt{n}} \theta y + \frac{1}{2}\left( \frac{\sqrt{n+1} - \sqrt{n} -1}{\sqrt{n}} \theta y \right)^2 \right) p_1(1 - \gamma_n y)\\
&\quad\leq \left(-\alpha_n \theta y + \frac{\alpha_n^2}{2\sqrt{n+1}} \theta^2  y^2 \right)p_1(1 - \gamma_n y)\leq\theta y\left(-\alpha_n + \frac{\alpha_n^2}2 \theta  \delta \right)p_1(1 - \gamma_n y)\\
&\quad\leq\left(\varepsilon+\left(-1+\frac{\theta\delta}2\right)\right)\theta y\quad\text{for $n$ large.}
 \end{align*}
where $\alpha_n = \left(1- \sqrt{n+1} + \sqrt{n}\right) \gamma_n\gamma_{n+1}^{-1}$. There exists $\xi^{(\delta)}$, such that $1-\delta\leq \xi^{(\delta)}\leq1$ and the second term writes:
\begin{align*}
&\sqrt{n+1} \left(\exp\left(\theta+\frac{\sqrt{n+1} - \sqrt{n} - 1}{\sqrt{n}} \theta y \right) - 1\right) p_0(1 - \gamma_n y)\leq\sqrt{n+1}p_0(1-\gamma_n y)(\e^\theta-1)\\
&\quad\leq-\sqrt{n+1}\gamma_nyp_0'(\xi^{(\delta)})(\e^\theta-1)\leq\left(\varepsilon-(\e^\theta-1) p_0'(1)\right) y\quad\text{for $n$ large.}
 \end{align*}
The third term is negative, and the fourth term writes:
\begin{align*}
&\sqrt{n+1}\left(\exp\left(\frac\theta{\sqrt{n+1}}+\frac{n- \sqrt{n(n+1)}}{\sqrt{n(n+1)}} \theta y\right) - 1\right)\widetilde p_0(1 - \gamma_n y)\\
&\quad \leq\sqrt{n+1}\left(\exp\left(\frac\theta{\sqrt{n+1}}\right) - 1\right)\leq\theta+\varepsilon\quad\text{for $n$ large.}
\end{align*}
Hence, there exists some (deterministic) $n_0\geq l$ such that, for $n\geq n_0$,
\[\mathcal L_n^{(l,\delta)}V(y)\leq V(y)\left[\theta+\varepsilon-y\left(p_0'(1)(\e^\theta-1)-\left(\theta+\frac{\theta\delta}2\right)p_1(1)+\epsilon(1+\theta)\right)\right].\]
Then, for $\varepsilon,\delta,\theta$ small enough, there exists $\widetilde\alpha>0$ such that, for $n\geq n_0$ and for any $M\geq\widetilde(\theta+\epsilon)\alpha^{-1}$,
\[\mathcal L_n^{(l,\delta)}V(y)\leq V(y)(\theta+\varepsilon-\widetilde\alpha y)\leq-(\widetilde\alpha M-\theta-\varepsilon)V(y)+\widetilde\alpha MV(M).\]
Then, Assumption~\ref{assumption:moments}.iii holds with
\[\alpha=\left(p_0'(1)(\e^\theta-1)-\left(\theta+\frac{\theta\delta}2\right)p_1(1)+\epsilon(1+\theta)\right)M-\theta-\varepsilon,\quad\beta=\widetilde\alpha MV(M).\]

Finally, checking Assumption~\ref{assumption:variance} is easy (using \eqref{eq:proofTruncatedChain1} for instance) with $d_2=3$, which forces us to set $d=6$ (since $\Gamma_n\chi_3\leq M_6\chi_6$). The chain $(y_n^{(l,\delta)})_{n\geq0}$ satisfying a Lyapunov criterion with $V(y)=\e^{\theta y}$, its moments of order 6 are also uniformly bounded.
\end{proof}

\subsection{Appendix for the decreasing step Euler scheme}

\begin{proof}[Proof of Lemma~\ref{lemma:DerivativesDiffusion}]
Applying It\^o's formula with $x\mapsto |x|^p$, we get
\begin{align}
|\partial_xX_t^x|^p&=1+\int_0^tp\left(b'(X_s^x)|\partial_xX_s^x|^p+ \frac{p-1}2(\sigma'(X_s^x))^2|\partial_xX_s^x|^p\right)ds\notag\\
&\quad+\int_0^tp\sigma'(X_s^x)|\partial_xX_s^x|^pdW_s\notag\\
&\leq1+C\int_0^t|\partial_xX_s^x|^pds+\int_0^tp\sigma'(X_s^x)|\partial_xX_s^x|^pdW_s,\label{eq:SDE_Yp}
\end{align}
where $C=p\|b'\|_\infty+\frac{p(p-1)}2\|\sigma'\|_\infty^2$. Let us show that $\int_0^tp\sigma'(X_s^x)|\partial_xX_s^x|^pdW_s$ is a martingale. To that end, since $|\partial_xX_t^x|^p$ is non-negative and $(x+y+z)^2\leq2(x^2+y^2+z^2)$, we use the Burkholder–Davis–Gundy's inequality so there exists a constant $C'$ such that,
\begin{align*}
|\partial_xX_t^x|^p&\leq1+C\int_0^t\sup_{u\in[0,s]}|\partial_xX_u^x|^pds+\int_0^tp\sigma'(X_s^x)|\partial_xX_s^x|^pdW_s\\
\sup_{u\in[0,t]}|\partial_xX_u^x|^p&\leq1+C\int_0^t\sup_{u\in[0,s]}|\partial_xX_u^x|^pds+\sup_{u\in[0,t]}\int_0^up\sigma'(X_s^x)|\partial_xX_s^x|^pdW_s\\
\E\left[\sup_{u\in[0,t]}|\partial_xX_u^x|^{2p}\right]&\leq2+2C^2T\int_0^t\E\left[\sup_{u\in[0,s]}|\partial_xX_u^x|^{2p}\right]ds\\
&\quad+2\E\left[\left(\sup_{u\in[0,t]}\int_0^up\sigma'(X_s^x)|\partial_xX_s^x|^pdW_s\right)^2\right]\\
&\leq2+2C^2T\int_0^t\E\left[\sup_{u\in[0,s]}|\partial_xX_u^x|^{2p}\right]ds+2C'\int_0^t\E[\sigma'(X_s^x)^2|\partial_xX_s^x|^{2p}]ds\\
&\leq2+2C^2T\int_0^t\E\left[\sup_{u\in[0,s]}|\partial_xX_u^x|^{2p}\right]ds\\
&\quad+2C'\|\sigma'\|_\infty^2\int_0^t\E\left[\sup_{u\in[0,s]}|\partial_xX_u^x|^{2p}\right]ds\\
&\leq2\exp\left((C^2T+C'\|\sigma'\|_\infty^2)T\right)\quad\text{by Grönwall's Lemma}.
\end{align*}
Hence, $\int_0^tp\sigma'(X_s^x)|\partial_xX_s^x|^pdW_s$ is a martingale and, taking the expected values in \eqref{eq:SDE_Yp} and applying Grönwall's lemma once again, we have
\[\E[|\partial_xX_t^x|^p]\leq\exp\left(\left(p\|b'\|_\infty+\frac{p(p-1)}2\|\sigma'\|_\infty^2\right)T\right).\]
Using Hölder's inequality for $p=2$ completes the case of the first derivative.

Since the following computations are more and more tedious, we choose to treat only the case of the second derivative. Note that $\partial_x^2X^x_t$ exists and satisfies the following SDE:
\begin{align*}
\partial_x^2X^x_t&=\int_0^t\left(b'(X^x_s)\partial_x^2X^x_s +b''(X^x_s)(\partial_xX^x_s)^2\right)ds\\
&\quad+\int_0^t\left(\sigma'(X^x_s)\partial_x^2X^x_s+\sigma''(X^x_s)(\partial_xX^x_s)^2\right)dW_s.
\end{align*}
Itô's formula provides us the following inequation:
\begin{align*}
|\partial_x^2X^x_t|^p&\leq C_1\int_0^t|\partial_x^2X^x_s|^pds+C_2\int_0^t|\partial_x^2X^x_s|^{p-1}|\partial_xX^x_s|^2ds+C_3\int_0^t|\partial_x^2X^x_s|^{p-2}|\partial_xX^x_s|^4ds\\
&\quad+\int_0^tp\bigg(|\partial_x^2X^x_s|^p\sigma'(X^x_s)+|\partial_x^2X^x_s|^{p-1}\sgn(\partial^2_xX^x_s)\sigma''(X^x_s)|\partial_xX^x_s|^2\bigg)dW_s,
\end{align*}
with constants $C_i$ depending on $p,\|b'\|_\infty,\|b''\|_\infty,\|\sigma'\|_\infty,\|\sigma''\|_\infty$. The last term proves to be a martingale, with similar arguments as above. We take the expected values, and apply Hölder's inequality twice to find, for $p>2$,
\begin{align*}
\E\Big[|\partial_x^2X^x_t|^p\Big]&\leq C_1\int_0^t\E\Big[|\partial_x^2X^x_s|^p\Big]ds +C_2\int_0^t\E\Big[|\partial_x^2X^x_s|^{p-1}|\partial_xX^x_s|^2\Big]ds\\
&\quad+C_3\int_0^t\E\Big[|\partial_x^2X^x_s|^{p-2}|\partial_xX^x_s|^4\Big]ds\\
&\leq C_1\int_0^t\E\Big[|\partial_x^2X^x_s|^p\Big]ds +C_2\int_0^t\E\Big[|\partial_x^2X^x_s|^p\Big]^{\frac{p-1}p}\E\Big[|\partial_xX^x_s|^{2p}\Big]^{\frac1p}ds\\
&\quad+C_3\int_0^t\E\Big[|\partial_x^2X^x_s|^p\Big]^{\frac{p-2}p}\E\Big[|\partial_xX^x_s|^{2p}\Big]^{\frac2p}ds\\
&\leq C_3\e^{C_4T}+C_1\int_0^t\E\Big[|\partial_x^2X^x_s|^p\Big]ds +(C_2+C_3)\e^{C_4T}\int_0^t\E\Big[|\partial_x^2X^x_s|^p\Big]^{\frac{p-1}p}ds,
\end{align*}
with $C_4=4\|b'\|_\infty+2(p-1)\|\sigma'\|_\infty^2$. The case $p=2$ is deduced straightforwardly:
\[\E\Big[|\partial_x^2X^x_t|^2\Big]\leq C_3\e^{C_4T}+C_1\int_0^t\E\Big[|\partial_x^2X^x_s|^2\Big]ds +C_3\e^{C_4T}\int_0^t\E\Big[|\partial_x^2X^x_s|^2\Big]^{\frac12}ds.\]
Regardless, since the unique solution of $u=Au+Bu^\alpha$ is
\[u(t)=\left(\left(u(0)^{1-\alpha}+\frac BA\right)\exp(A(1-\alpha)t)-\frac BA\right)^\frac1{1-\alpha},\]
for $A,B>0,\alpha\in(0,1),u(0)>0$, we have
\begin{align*}
\E\Big[|\partial_x^2X^x_t|^2\Big]&\leq \left(\left(C_2^\frac1p\e^{\frac{C_4}pT}+\frac{C_2+C_3}{C_1}\e^{C_4T}\right)\e^{\frac{C_1}pT}-\frac{C_2+C_3}{C_1}\e^{C_4T}\right)^p\\
&\leq \left(C_2^\frac1p\e^{\frac{C_4}pT}+\frac{C_2+C_3}{C_1}\e^{C_4T}\right)^p\e^{C_1T}.
\end{align*}
The same reasoning for the third derivative achieves the proof.
\end{proof}
\begin{remark}[Regularity of general diffusion processes]
The quality of approximation of a diffusion process is not completely unrelated to its regularity, see for instance \cite[Theorem~1.3]{HHJ15}. In higher dimension, smoothness is generally checked under Hörmander conditions (see e.g. \cite{Hai11,HHJ15}).
\end{remark}

\begin{acknowledgements}
The authors would like to thank Jean-Christophe Breton, Florent Malrieu, Eva Löcherbach and the referees for their attentive reading and comments, as well as Pierre Monmarché for redactional issues. This work was financially supported by the ANR PIECE (ANR-12-JS01-0006-01), the SNF (grant 149871), the Chair \emph{Modélisation Mathématique et Biodiversité}, and an outgoing mobility grant from the Université Européenne de Bretagne. This article is part of the Ph.D. thesis of F.B., which is supported by the Centre Henri Lebesgue (programme "Investissements d'avenir" ANR-11-LABX-0020-01).
\end{acknowledgements}

\bibliography{Biblio}

\end{document}